\documentclass{amsproc}

\usepackage{verbatim}
\usepackage{xcolor}
\usepackage{tikz}
\usepackage{tikz-cd}
\usepackage{arydshln}
\usepackage{caption}
\usepackage{subcaption}
\usepackage{graphicx}
\tikzset{
	symbol/.style={
		draw=none,
		every to/.append style={
			edge node={node [sloped, allow upside down, auto=false]{$#1$}}}
	}
}

\usepackage{amsmath}
\usepackage{bbm}
\usepackage{hyperref}

\newtheorem{theorem}{Theorem}[section]
\newtheorem{proposition}[theorem]{Proposition}
\newtheorem{lemma}[theorem]{Lemma}
\newtheorem{corollary}[theorem]{Corollary}

\theoremstyle{definition}

\theoremstyle{remark}
\newtheorem{remark}[theorem]{Remark}
\newtheorem{question}[theorem]{Question}

\numberwithin{equation}{section}

\begin{document}

\title[Equidistribution of horospheres on moduli spaces]{Equidistribution of horospheres on moduli spaces of hyperbolic surfaces}



\author{Francisco Arana--Herrera}

\email{farana@stanford.edu}

\address{Department of Mathematics, Stanford University, 450 Jane Stanford Way, Building 380, Stanford, CA 94305-2125, USA.}



\date{}

\begin{abstract}
Given a simple closed curve $\gamma$ on a connected, oriented, closed surface $S$ of negative Euler characteristic, Mirzakhani showed that the set of points in the moduli space of hyperbolic structures on $S$ having a simple closed geodesic of length $L$ of the same topological type as $\gamma$ equidistributes with respect to a natural probability measure as $L \to \infty$. We prove several generalizations of Mirzakhani's result and discuss some of the technical aspects ommited in her original work. The dynamics of the earthquake flow play a fundamental role in the arguments in this paper.
\end{abstract}

\maketitle


\thispagestyle{empty}

\tableofcontents

\section{Introduction}

$ $

Let $X$ be a connected, oriented, complete, finite area hyperbolic surface with $n \geq 1$ cusps. On the unit tangent bundle $T^1X$ one can find $n$ one-parameter families of periodic horocycle flow orbits, called \textit{horocycles}, one for each cusp of $X$; the orbits in each family are indexed by $\mathbf{R}_{>0}$ according to their period. In \cite{Sar81}, Sarnak showed that each one of these families of horocycles equidistributes with respect to the Liouville measure on $T^1X$ as the period of the horocycles goes to infinity. His methods, which involve the use of Einsenstein series, also give a precise description of the equidistribution rate.\\

Applying techniques introduced by Margulis in his thesis, see \cite{Mar04} for an English translation, one can use the exponential mixing property of the geodesic flow to obtain analogous equidistribution results for families of \textit{horospheres} on non-compact, finite volume, rank-$1$, locally symmetric spaces; see \S2 in \cite{KM96} for a detailed explanation of closely related ideas. If one is willing to forgo the information about the equidistribution rates of such families, the same results can also be proved using Dani's classification of all horospherical flow invariant and ergodic probability measures; see \cite{Dani81}. \\

Marguli's techniques have also been successfully applied to prove analogous equidistribution results on non-locally-symmetric, non-homogeneous spaces; see \cite{EMM19} for an example of such an application to moduli spaces of Riemann surfaces. Similar techniques have also been used to understand other (non-horosphere-like) equidistribution phenomena both in finite volume, rank-1, locally symmetric spaces, see \cite{EM93}, as well as in non-locally-symmetric, non-homogeneous spaces, see \cite{ABEM12} for an example in the context of moduli spaces of Riemann surfaces. \\

Surprisingly enough, similar equidistribution phenomena can also be observed on non-locally-symmetric, non-homogeneous spaces which are not known to carry a ``mixing geodesic flow" and for which no classification of the ``horospherical flow invariant and ergodic measures" is known. For example, in \cite{Mir07b}, Mirzakhani used the ergodicity and strong non-divergence properties of the earthquake flow, see \cite{Mir08a} and \cite{MW02}, respectively, to show that, given a simple closed curve $\gamma$ on a connected, oriented, closed surface $S$, the family of level sets on the moduli space of hyperbolic structures on $S$ on which the curve $\gamma$ has length $L$ equidistributes with respect to a natural probability measure as $L \to \infty $. A more general version of this result plays a fundamental role in Mirzakhani's work on counting problems for filling closed geodesics on hyperbolic surfaces, see \cite{Mir16}.  \\

The main goal of this paper is to prove several generalizations of Mirzakhani's result while also discussing some of the technical aspects that are ommited in her original work. Such generalizations play a crucial role in future work of the author, see \cite{AH19b} and Theorem \ref{theo:count} below.\\

\textit{Main results.} Let $g,n \in \mathbf{Z}_{\geq 0}$ be a pair of non-negative integers satisfying $2 - 2g - n < 0$.  For the rest of this paper consider a fixed connected, oriented surface $S_{g,n}$ of genus $g$ with $n$ punctures (and negative Euler characteristic). Unless otherwise specified, the term \textit{length} will always refer to \textit{hyperbolic length}.\\

Let $\mathcal{T}_{g,n}$ be the Teichmüller space of marked, oriented, complete, finite area hyperbolic structures on $S_{g,n}$. Let $\mathcal{ML}_{g,n}$ be the space of (compactly supported) measured geodesic laminations on $S_{g,n}$. Over $\mathcal{T}_{g,n}$ we consider the bundle $P^1 \mathcal{T}_{g,n}$ of unit length measured geodesic laminations. More concretely,
\[
P^1 \mathcal{T}_{g,n} := \{(X,\lambda) \in \mathcal{T}_{g,n} \times \mathcal{ML}_{g,n} \ | \ \ell_{\lambda}(X) = 1\},
\]
where $\ell_{\lambda}(X)>0$ denotes the hyperbolic length of the measured geodesic lamination $\lambda \in \mathcal{ML}_{g,n}$ on the marked hyperbolic structure $X \in \mathcal{T}_{g,n}$. The bundle $P^1 \mathcal{T}_{g,n}$ carries a natural measure $\nu_{\text{Mir}}$, called the \textit{Mirzakhani measure}, which disintegrates as the Weil-Petersson measure $\mu_{\text{wp}}$ of $\mathcal{T}_{g,n}$ on the base and as a coned-off version of the Thurston measure $\mu_{\text{Thu}}$ of $\mathcal{ML}_{g,n}$ on the fibers; see \S 2 for a precise definition.\\

The mapping class group of $S_{g,n}$, denoted $\text{Mod}_{g,n}$, acts diagonally on $P^1 \mathcal{T}_{g,n}$ in a properly discontinuous way preserving the Mirzakhani measure. The quotient $P^1 \mathcal{M}_{g,n} := P^1\mathcal{T}_{g,n}/\text{Mod}_{g,n}$ is the bundle of unit length measured geodesic laminations over the moduli space $\mathcal{M}_{g,n}$ of oriented, complete, finite area hyperbolic structures on $S_{g,n}$. Locally pushing forward the measure $\nu_{\text{Mir}}$ on $P^1\mathcal{T}_{g,n}$ under the quotient map $P^1\mathcal{T}_{g,n} \to P^1\mathcal{M}_{g,n}$ yields a natural measure $\widehat{\nu}_{\text{Mir}}$ on $P^1\mathcal{M}_{g,n}$, also called the Mirzakhani measure. The total mass of $P^1 \mathcal{M}_{g,n}$ with respect to $\widehat{\nu}_{\text{Mir}}$ is finite, see Theorem 3.3 in \cite{Mir08b}. We denote it by 
\begin{equation}
\label{eq:b_gn}
b_{g,n} := \widehat{\nu}_{\text{Mir}}(P^1 \mathcal{M}_{g,n}) < +\infty.
\end{equation}
$ $

Let $\gamma := (\gamma_1,\dots,\gamma_k)$ with $1 \leq k \leq 3g-3+n$ be an ordered tuple of pairwise non-isotopic, pairwise disjoint, simple closed curves on $S_{g,n}$, \textit{ordered simple closed multi-curve} for short, and $f \colon (\mathbf{R}_{\geq 0})^k \to \mathbf{R}_{\geq0}$ be a bounded, compactly supported, Borel measurable function with non-negative values and which is not almost everywhere zero with respect to the Lebesgue measure class. For every $L > 0$ we consider a \textit{horoball segment} $B_\gamma^{f,L} \subseteq \mathcal{T}_{g,n}$ given by 
\[
B_\gamma^{f,L} := \{X \in \mathcal{T}_{g,n} \ | \ (\ell_{\gamma_i}(X))_{i=1}^k \in L \cdot \text{supp}(f) \}.
\]
Every such horoball segment supports a natural \textit{horoball segment measure} $\mu_\gamma^{f,L}$ defined as 
\[
d \mu_\gamma^{f,L}(X) := f\left(\textstyle \frac{1}{L} \cdot (\ell_{\gamma_i}(X))_{i=1}^k\right) \ d\mu_{\text{wp}}(X).
\]  
To get a meaningful (locally finite) horoball segment measure on moduli space, we need to get rid of the redundancies that show up when taking pushforwards. A natural way to do this is to consider the intermediate cover
\[
\mathcal{T}_{g,n} \to \mathcal{T}_{g,n}/\text{Stab}(\gamma) \to \mathcal{M}_{g,n},
\]
where
\[
\text{Stab}(\gamma) = \bigcap_{i=1}^k \text{Stab}(\gamma_i) \subseteq \text{Mod}_{g,n}
\]
is the subgroup of all mapping classes of $S_{g,n}$ that preserves every component of $\gamma$ up to isotopy. The measure $\mu_\gamma^{f,L}$ on $\mathcal{T}_{g,n}$ is $\text{Stab}(\gamma)$-invariant. Let $\widetilde{\mu}_\gamma^{f,L}$ be the local pushforward of $\mu_\gamma^{f,L}$ to $\mathcal{T}_{g,n}/\text{Stab}(\gamma)$ and $\widehat{\mu}_{\gamma}^{f,L}$ be the pushforward of $\widetilde{\mu}_\gamma^{f,L}$ to $\mathcal{M}_{g,n}$. \\

Let $\mathcal{ML}_{g,n}(\mathbf{Q}) \subseteq \mathcal{ML}_{g,n}$ be the set of all unordered, positively rationally weighted, simple closed multi-curves on $S_{g,n}$. Let $\mathbf{a}:=(a_1,\dots,a_k) \in (\mathbf{Q}_{>0})^k$ be a vector of positive rational weights on the components of $\gamma$. Denote 
\begin{equation}
\label{eq:weight_curve}
\mathbf{a} \cdot \gamma := a_1\gamma_1 + \cdots + a_k \gamma_k \in \mathcal{ML}_{g,n}(\mathbf{Q}).
\end{equation}
The horoball segment measures $\mu_\gamma^{f,L}$ on $\mathcal{T}_{g,n}$ also give rise to horoball segment measures $\nu_{\gamma,\mathbf{a}}^{f,L}$ on the bundle $P^1 \mathcal{T}_{g,n}$ by considering the disintegration formula
\[
d \nu_{\gamma,\mathbf{a}}^{f,L}(X,\lambda) := d \delta_{\mathbf{a} \cdot \gamma/ \ell_{\mathbf{a} \cdot \gamma}(X)}(\lambda) \ d\mu_\gamma^{f,L}(X),
\]
where the symbol $\delta$ is used to denote point masses. In analogy with the case above, to get locally finite horoball segment measures on $P^1\mathcal{M}_{g,n}$, we consider the intermediate cover
\[
P^1\mathcal{T}_{g,n} \to P^1\mathcal{T}_{g,n}/\text{Stab}(\gamma) \to P^1\mathcal{M}_{g,n}.
\]
Let $\widetilde{\nu}_{\gamma,\mathbf{a}}^{f,L}$ be the local pushforward of $\nu_{\gamma,\mathbf{a}}^{f,L}$ to $P^1\mathcal{T}_{g,n}/\text{Stab}(\gamma)$ and  $\widehat{\nu}_{\gamma,\mathbf{a}}^{f,L}$ the pushforward of $\widetilde{\nu}_{\gamma,\mathbf{a}}^{f,L}$ to $P^1\mathcal{M}_{g,n}$.\\

One can check, see Proposition \ref{prop:total_hor_meas}, that the measures $\widehat{\mu}_\gamma^{f,L}$ and $\widehat{\nu}_{\gamma,\mathbf{a}}^{f,L}$ are finite. We denote by $m_\gamma^{f,L}$ the total mass of the measures $\widehat{\mu}_\gamma^{f,L}$ and $\widehat{\nu}_{\gamma,\mathbf{a}}^{f,L}$, i.e.,
\[
m_\gamma^{f,L} := \widehat{\mu}_\gamma^{f,L}(\mathcal{M}_{g,n}) = \widehat{\nu}_{\gamma,\mathbf{a}}^{f,L}(P^1\mathcal{M}_{g,n}) <+\infty.
\]
$ $

One of the main results of this paper is the following theorem, which shows that horoball segment measures on $P^1\mathcal{M}_{g,n}$ equidistribute with respect to the Mirzakhani measure $\widehat{\nu}_{\text{Mir}}$.\\

\begin{theorem}
	\label{theo:horoball_equid}
	In the weak-$\star$ topology for measures on $P^1\mathcal{M}_{g,n}$,
	\[
	\lim_{L \to \infty} \frac{\widehat{\nu}_{\gamma,\mathbf{a}}^{f,L}}{m_\gamma^{f,L}} = \frac{\widehat{\nu}_{\text{Mir}}}{b_{g,n}}.
	\]
\end{theorem}
$ $

Consider the function $B \colon \mathcal{M}_{g,n} \to \mathbf{R}_{>0}$ which to every hyperbolic structure $X \in \mathcal{M}_{g,n}$ assigns the value
\begin{equation}
\label{eq:mir_fn}
B(X):= \mu_{\text{Thu}}(\{\lambda \in \mathcal{ML}_{g,n} \ | \ \ell_\lambda(X) \leq 1\}).
\end{equation}
We refer to this function as the \textit{Mirzakhani function}. Roughly speaking, $B(X)$ measures the shortness of closed geodesics on $X$. By work of Mirzakhani, see \cite{Mir08b}, $B$ is continuous and proper. Moreover, see \cite{AA19}, one can give upper and lower bounds of the same order describing the behavior of $B$ near the cusp of $\mathcal{M}_{g,n}$. \\

Let $\widehat{\mu}_\text{wp}$ be the pushforward to $\mathcal{M}_{g,n}$ of the Weil-Petersson measure $\mu_\text{wp}$ on $\mathcal{T}_{g,n}$. Taking pushforwards under the bundle map $P^1\mathcal{M}_{g,n} \to \mathcal{M}_{g,n}$ in the statement of Theorem \ref{theo:horoball_equid}, we deduce the following important corollary, which shows that horoball segment measures on $\mathcal{M}_{g,n}$ equidistribute with respect to $B(X) \cdot d\widehat{\mu}_{wp}$. \\

\begin{corollary}
	\label{cor:horoball_equid}
	In the weak-$\star$ topology for measures on $\mathcal{M}_{g,n}$,
	\[
	\lim_{L \to \infty} \frac{\widehat{\mu}_\gamma^{f,L}}{m_\gamma^{f,L}} = \frac{B(X) \cdot d\widehat{\mu}_{\text{wp}}(X)}{b_{g,n}}.
	\]
\end{corollary}
$ $

We will also be concerned with the equidistribution on $\mathcal{M}_{g,n}$ of measures supported on codimension 1 analogues of horoball segments. Let $\gamma := (\gamma_1,\dots,\gamma_k)$ with $1 \leq k \leq 3g-3+n$ be an ordered simple closed multi-curve on $S_{g,n}$ and $\mathbf{a}:= (a_1,\dots,a_k) \in (\mathbf{Q}_{>0})^k$ be a vector of positive rational weights on the components of $\gamma$. Denote by $\Delta_\mathbf{a} \subseteq (\mathbf{R}_{\geq0})^k$ the compact, codimension 1 simplex 
\[
\Delta_\mathbf{a} := \{(x_1,\dots,x_k) \in (\mathbf{R}_{\geq0})^k \ | \ a_1x_1 + \cdots + a_k x_k = 1 \}.
\]
For every $L > 0$ we consider a \textit{horosphere}  $S_{\gamma,\mathbf{a}}^{L} \subseteq \mathcal{T}_{g,n}$ given by 
\[
S_{\gamma,\mathbf{a}}^{L} := \{X \in \mathcal{T}_{g,n} \ | \ (\ell_{\gamma_i}(X))_{i=1}^k \in L \cdot \Delta_{\mathbf{a}} \}.
\]
Every such horosphere supports a natural \textit{horosphere measure} $\eta_{\gamma,\mathbf{a}}^{L}$ defined in the following way. Let $\mathbf{a} \cdot \gamma \in \mathcal{ML}_{g,n}(\mathbf{Q})$ be as in (\ref{eq:weight_curve}) and
\[
\ell_{\mathbf{a} \cdot \gamma} \colon \mathcal{T}_{g,n} \to \mathbf{R}_{>0}
\]
be the function which to every marked hyperbolic structure $X \in \mathcal{T}_{g,n}$ assigns the total hyperbolic length of the unique multi-geodesic representative of $\mathbf{a} \cdot \gamma$ on $X$. The function $\ell_{\mathbf{a} \cdot \gamma} \colon \mathcal{T}_{g,n} \to \mathbf{R}_{>0}$ is smooth, surjective, and non-singular (i.e., it has no critical points). It follows that the horosphere
\[
S_{\gamma,\mathbf{a}}^{L} = \ell_{\mathbf{a} \cdot \gamma}^{-1}(L)
\]
is a smoothly embedded codimension $1$ real submanifold of $\mathcal{T}_{g,n}$. The Weil-Petersson volume form $\omega_\text{wp}$ on $\mathcal{T}_{g,n}$ induces a natural volume form $\omega_{\gamma,\mathbf{a}}^{L}$ on $S_{\gamma,\mathbf{a}}^{L}$ by contraction. More precisely, $\omega_{\gamma,\mathbf{a}}^{L}$ is the volume form obtained by restricting to $S_{\gamma,\mathbf{a}}^{L}$ the contraction of $\omega_{\text{wp}}$ by any vector field $F$ on $\mathcal{T}_{g,n}$ satisfying $d\ell_{\mathbf{a} \cdot \gamma}(F) \equiv 1$. This definition is independent of the choice of vector field $F$; see \S 2 for details. The horosphere measure $\eta_{\gamma,\mathbf{a}}^{L}$ is the measure $|\omega_{\gamma,\mathbf{a}}^{L}|$ induced by the volume form $\omega_{\gamma,\mathbf{a}}^{L}$ on $S_{\gamma,\mathbf{a}}^{L}$. We also denote by $\eta_{\gamma,\mathbf{a}}^{L}$ the extension by zero of this measure to $\mathcal{T}_{g,n}$. This measure is $\text{Stab}(\gamma)$-invariant.\\

Let $f \colon \Delta_{\mathbf{a}} \to \mathbf{R}_{\geq0}$ be a bounded, Borel measurable function with non-negative values and which is not almost everywhere zero with respect to the Lebesgue measure class. For every $L > 0$ we consider the \textit{horosphere sector} $S_{\gamma,\mathbf{a}}^{f,L} \subseteq S_{\gamma,\mathbf{a}}^L \subseteq \mathcal{T}_{g,n}$ given by 
\[
S_{\gamma,\mathbf{a}}^{f,L} := \{X \in \mathcal{T}_{g,n} \ | \ (\ell_{\gamma_i}(X))_{i=1}^k \in L \cdot (\Delta_{\mathbf{a}} \cap \text{supp}(f)) \}.
\]
Every such horosphere sector supports a natural \textit{horosphere sector measure} $\eta_{\gamma,\mathbf{a}}^{f,L}$ defined as 
\[
d \eta_{\gamma,\mathbf{a}}^{f,L}(X) := f\left(\textstyle \frac{1}{L} \cdot (\ell_{\gamma_i}(X))_{i=1}^k\right) \ d\eta_{\gamma,\mathbf{a}}^L(X).
\]  
This measure is $\text{Stab}(\gamma)$-invariant. In analogy with the case of horoball segment measures, let $\widetilde{\eta}_{\gamma,\mathbf{a}}^{f,L}$ be the local pushforward of $\eta_{\gamma,\mathbf{a}}^{f,L}$ to $\mathcal{T}_{g,n}/\text{Stab}(\gamma)$ and $\widehat{\eta}_{\gamma,\mathbf{a}}^{f,L}$ be the pushforward of $\widetilde{\eta}_{\gamma,\mathbf{a}}^{f,L}$ to $\mathcal{M}_{g,n}$.\\

The horosphere sector measures $\eta_{\gamma,\mathbf{a}}^{f,L}$ on $\mathcal{T}_{g,n}$ also give rise to horosphere sector measures $\zeta_{\gamma,\mathbf{a}}^{f,L}$ on the bundle $P^1 \mathcal{T}_{g,n}$ by considering the disintegration formula
\[
d \zeta_{\gamma,\mathbf{a}}^{f,L}(X,\lambda) := d \delta_{\mathbf{a} \cdot \gamma/ \ell_{\mathbf{a} \cdot \gamma}(X)}(\lambda) \ d\eta_{\gamma,\mathbf{a}}^{f,L}(X).
\]
In analogy with the case of horoball segment measures, let $\widetilde{\zeta}_{\gamma,\mathbf{a}}^{f,L}$ be the local pushforward of $\zeta_{\gamma,\mathbf{a}}^{f,L}$ to $P^1\mathcal{T}_{g,n}/\text{Stab}(\gamma)$ and  $\widehat{\zeta}_{\gamma,\mathbf{a}}^{f,L}$ the pushforward of $\widetilde{\zeta}_{\gamma,\mathbf{a}}^{f,L}$ to $P^1\mathcal{M}_{g,n}$.\\

One can check, see Proposition \ref{prop:total_horosphere_meas}, that the measures $\widehat{\eta}_{\gamma,\mathbf{a}}^{f,L}$ and $\widehat{\zeta}_{\gamma,\mathbf{a}}^{f,L}$ are finite. We denote by $n_{\gamma,\mathbf{a}}^{f,L}$ the total mass of the measures $\widehat{\eta}_{\gamma,\mathbf{a}}^{f,L}$ and $\widehat{\zeta}_{\gamma,\mathbf{a}}^{f,L}$, i.e.,
\[
n_{\gamma,\mathbf{a}}^{f,L} := \widehat{\eta}_{\gamma,\mathbf{a}}^{f,L}(\mathcal{M}_{g,n}) = \widehat{\zeta}_{\gamma,\mathbf{a}}^{f,L}(P^1\mathcal{M}_{g,n}) < +\infty.
\]
$ $

Another one of the main results of this paper is the following theorem, which shows that horosphere sector measures on $P^1\mathcal{M}_{g,n}$ equidistribute with respect to the Mirzakhani measure $\widehat{\nu}_{\text{Mir}}$.\\

\begin{theorem}
	\label{theo:horosphere_equid}
	In the weak-$\star$ topology for measures on $P^1\mathcal{M}_{g,n}$,
	\[
	\lim_{L \to \infty} \frac{\widehat{\zeta}_{\gamma,\mathbf{a}}^{f,L}}{n_{\gamma,\mathbf{a}}^{f,L}} = \frac{\widehat{\nu}_{\text{Mir}}}{b_{g,n}}.
	\]
\end{theorem}
$ $

Taking pushforwards under the bundle map $P^1\mathcal{M}_{g,n} \to \mathcal{M}_{g,n}$ in the statement of Theorem \ref{theo:horosphere_equid}, we deduce the following important corollary.\\

\begin{corollary}
	\label{cor:horosphere_equid}
	In the weak-$\star$ topology for measures on $\mathcal{M}_{g,n}$,
	\[
	\lim_{L \to \infty} \frac{\widehat{\eta}_{\gamma,\mathbf{a}}^{f,L}}{n_{\gamma,\mathbf{a}}^{f,L}} = \frac{B(X) \cdot d\widehat{\mu}_{\text{wp}}(X)}{b_{g,n}}.
	\]
\end{corollary}
$ $

\begin{remark}
	Mirzakhani's original equidistribution results in \cite{Mir07b} can be recovered from Theorem \ref{theo:horosphere_equid} and Corollary \ref{cor:horosphere_equid} by letting $\gamma$ be a connected simple closed curve on $S_{g,n}$, $\mathbf{a} := (1)$, and $f \colon \Delta_{(1)} \to \mathbf{R}_{>0}$ be the constant function taking the value $1$ on $\Delta_{(1)} := \{1\}$.
\end{remark}
$ $

Although Theorem \ref{theo:horoball_equid} can be proved using Theorem \ref{theo:horosphere_equid}, we do not take such approach in this paper; most of the important ideas in the proof of Theorem \ref{theo:horosphere_equid} are already present in the proof Theorem \ref{theo:horoball_equid} and in a much more clear way. We only give a brief overview of the proof of Theorem \ref{theo:horosphere_equid} and focus our attention on proving Theorem \ref{theo:horoball_equid} in complete detail.\\

There is one important idea exclusive to the proof of Theorem \ref{theo:horosphere_equid}. We denote by $U_X(\epsilon) \subseteq \mathcal{T}_{g,n}$ the open ball of radius $\epsilon > 0$ centered at $X \in \mathcal{T}_{g,n}$ with respect to the symmetric Thurston metric on Teichmüller space; see \S 2 for a definition. The following bound, which we refer to as the \textit{Mirzakhani bound}, is a key ingredient in the proof of Theorem \ref{theo:horosphere_equid}.\\
	
\begin{proposition}
	\label{prop:mir_bd_og_0}
	Let $K \subseteq \mathcal{T}_{g,n}$ be a compact subset. There exist constants $C > 0$ and $\epsilon_0 > 0$ such that for every $X \in K$, every $0 < \epsilon < \epsilon_0$, every $\alpha \in \text{Mod}_{g,n} \cdot \gamma$, and every $L >0$, 
	\[
	\eta_{\alpha,\mathbf{a}}^L(U_X(\epsilon)) \leq C \cdot \frac{\epsilon^{6g-6+2n-1}}{L}.
	\]
\end{proposition}
$ $

Proposition \ref{prop:mir_bd_og_0} corresponds to Theorem 5.5 (b) in \cite{Mir07b}. In light of the author's inability to understand the proof there provided, see the footnote in \S 13.3 of \cite{Wri19} for a discussion on the matter, we give a conceptual proof different in nature from the one in Mirzakhani's original work. We actually prove a stronger result; see Proposition \ref{prop:mir_bd_og} below. The basic idea of the proof is that the measure of the intersection $S_{\alpha,\mathbf{a}}^L \cap U_X(\epsilon)$ can be compared to the measure of the ball $U_X(2\epsilon)$ by translating the intersection $S_{\alpha,\mathbf{a}}^L \cap U_X(\epsilon)$ along the gradient flow lines of the function $\ell_{\mathbf{a} \cdot \alpha} \colon \mathcal{T}_{g,n} \to \mathbf{R}_{>0}$ to generate a considerable amount of nearby parallel copies of such intersection that are contained in $U_X(2\epsilon)$. The compactness of the space $P\mathcal{ML}_{g,n}$ of projective measured geodesic laminations on $S_{g,n}$ provides the uniform control needed in the estimates.\\

\textit{Random hyperbolic surfaces.} Corollary \ref{cor:horoball_equid} together with Mirzakhani's integration formulas in \cite{Mir07a}, see \S 2 for a summary of the relevant results, yield a natural procedure for building random hyperbolic surfaces. \\

Let $\gamma := (\gamma_1,\dots,\gamma_k)$ with $1 \leq k \leq 3g-3+n$ be an ordered simple closed multi-curve on $S_{g,n}$ and $f \colon (\mathbf{R}_{\geq 0})^k \to \mathbf{R}_{\geq0}$ be a bounded, compactly supported, Borel measurable function with non-negative values and which is not almost everywhere zero with respect to the Lebesgue measure class. Let $S_{g,n}(\gamma)$ be the (potentially disconnected) surface with boundary obtained by cutting $S_{g,n}$ along the components of $\gamma$. Denote by $\Sigma_j$ with $j \in \{1,\dots,c\}$ the connected components of $S_{g,n}(\gamma)$. For every $j \in \{1,\dots,c\}$ let $b_j \in \mathbf{Z}_{\geq 0}$ be the number of boundary components of $\Sigma_j$. Given a vector $\mathbf{L}_j \in (\mathbf{R}_{>0})^{b_j}$ let $\mathcal{M}(\Sigma_j,\mathbf{L}_j)$ be the moduli space of oriented, complete, finite area hyperbolic structures on $\Sigma_j$ with labeled geodesic boundary components whose lengths are given by $\mathbf{L}_j$.\\

For every $L > 0$ build a random hyperbolic surfaces $X \in \mathcal{M}_{g,n}$ according to the following procedure:\\

\begin{enumerate}
	\item Pick a vector $\mathbf{L} := (\ell_i)_{i=1}^k \in (\mathbf{R}_{>0})^k$ at random according to the law
	\[
	f\left(\textstyle \frac{1}{L} \cdot (\ell_i)_{i=1}^k \right) \ d\ell_1 \cdots d\ell_k.
	\]
	For every $j \in \{1,\dots,c\}$ let $\mathbf{L}_j \in (\mathbf{R}_{>0})^{b_j}$ be the vector recording the entries of $\mathbf{L}$ corresponding to the boundary components of $\Sigma_j$.\\
	\item For every $j \in \{1,\dots,c\}$ pick a hyperbolic surface $X_j \in \mathcal{M}(\Sigma_j,\mathbf{L}_j)$ at random according to the Weil-Petersson volume form of such moduli space.\\
	\item Uniformly at random, pick a vector 
	\[
	\mathbf{t} := (t_i)_{i=1}^k \in\prod_{i=1}^k (\mathbf{R}/ \ell_i\mathbf{R}).
	\]
	Glue the hyperbolic surfaces $(X_j)_{j=1}^c$ along the components of $\gamma$ in an arbitrary way and then twist along such components according to the vector $\mathbf{t}$.\\
\end{enumerate}

By Theorem \ref{theo:cut_glue_fib}, this procedure yields a random hyperbolic surface $X \in \mathcal{M}_{g,n}$ whose probability law is precisely $\widehat{\mu}_\gamma^{f,L}$. According to Corollary \ref{cor:horoball_equid}, as $L \to \infty$, the probability law of such random hyperbolic surface converges in the weak-$\star$ topology to the probability law
\[
\frac{B(X) \cdot d\widehat{\mu}_{\text{wp}}(X)}{b_{g,n}}.
\]
$ $

Analogous intepretations also hold for horosphere sector measures.\\

\textit{Counting simple closed hyperbolic multi-geodesics with respect to the lengths of individual components.} In \cite{AH19b} we use Corollary \ref{cor:horoball_equid} together with a famous technique of Margulis to tackle several counting problems that generalize Mirzakhani's seminal work \cite{Mir08b}. The goal is to establish asymptotics for counting problems involving the lengths of individual components of simple closed hyperbolic multi-geodesics. In its most basic form, we prove a conjecture of Wolpert which we now describe.\\

Let $\gamma := (\gamma_1,\dots,\gamma_k)$ with $1 \leq k \leq 3g-3+n$ be an ordered simple closed multi-curve on $S_{g,n}$ and $\mathbf{b}:= (b_1,\dots,b_k) \in (\mathbf{R}_{>0})^k$ be arbitrary. For every $L > 0$ consider the counting function
\[
s(X,\gamma,\mathbf{b},L) := \# \{\alpha:=(\alpha_1,\dots,\alpha_k) \in \text{Mod}_{g,n} \cdot \gamma \ | \ \ell_{\alpha_i}(X) \leq b_i L, \ \forall i =1,\dots,k \}.
\]
Given $\mathbf{L}:=(\ell_1,\dots,\ell_k) \in (\mathbf{R}_{>0})^k$, let $W_{g,n}(\gamma,\mathbf{L})$ be the polynomial defined in \S 3; see the discussion preceding Proposition \ref{prop:total_hor_meas}. The following result was originally conjectured by Wolpert and is proved in \cite{AH19b}.\\

\begin{theorem}
	\label{theo:count}
	For every $\mathbf{b} := (b_1,\dots,b_k) \in (\mathbf{R}_{> 0})^k$,
	\[
	\lim_{L \to \infty} \frac{s(X,\gamma,\mathbf{b},L)}{L^{6g-6+2n}} = \frac{B(X)}{b_{g,n}} \cdot \int_{\prod_{i=1}^k [0,b_i]} W_{g,n}(\gamma,\mathbf{L}) \ d\mathbf{L},
	\]
	where $d\mathbf{L} := d\ell_1 \cdots d\ell_k$.
\end{theorem}
$ $

The proof of Theorem \ref{theo:count} in \cite{AH19b} is independent of Mirzakhani's work \cite{Mir08b} and recovers the counting results therein as a particular case. Letting $\mathbf{b} := (1,\dots,1) \in (\mathbf{R}_{> 0})^k$ one also obtains asymptotics for the counting functions
\[
m(X,\gamma,L) := \#\left\lbrace\alpha := (\alpha_1,\dots,\alpha_k) \in \text{Mod}_{g,n} \cdot \gamma \ \bigg\vert \ \max_{i=1,\dots,k} \ell_{\alpha_i}(X) \leq L \right\rbrace.
\]
$ $

\textit{Dynamics of the earthquake flow.} The dynamics of the earthquake flow 
\[
\{\text{tw}^t \colon P^1\mathcal{M}_{g,n} \to P^1\mathcal{M}_{g,n} \}_{t \in \mathbf{R}}
\]
play a crucial role in the proofs of Theorems \ref{theo:horoball_equid} and \ref{theo:horosphere_equid}; see \S 2 for a definition. By work of Mirzakhani, see \cite{Mir08a} or \cite{Wri18} for an expository account, the earthquake flow is ergodic in the following sense.\\

\begin{theorem}
	\label{theo:earth_erg}
	The measure $\widehat{\nu}_{\text{Mir}}$ on $P^1 \mathcal{M}_{g,n}$ is invariant and ergodic with respect to the earthquake flow.\\
\end{theorem}

Given $\epsilon > 0$, let $K_\epsilon$ be the \textit{$\epsilon$-thick part} of $\mathcal{M}_{g,n}$, i.e., the set of all oriented, complete, finite area hyperbolic structures on $S_{g,n}$ with no closed geodesics of length $<\epsilon$, and $P^1K_\epsilon$ be the \textit{$\epsilon$-thick part} of $P^1\mathcal{M}_{g,n}$, i.e., the union of all fibers of the bundle $P^1\mathcal{M}_{g,n}$ that lie above $K_\epsilon$. By Mumford's compactness criterion, see for example Theorem 12.6 in \cite{FM11}, these sets are compact. We denote by $\text{Leb}$ the Lebesgue measure on $\mathbf{R}$ and by $i(\cdot,\cdot)$ the geometric intersection pairing on $\mathcal{ML}_{g,n}$. By work of Minsky and Weiss, see Theorem E2 in \cite{MW02}, the earthquake flow on $P^1\mathcal{M}_{g,n}$ is non-divergent in the following strong sense.\\

\begin{theorem}
	\label{theo:mw_quant_rec}
	For every $\delta > 0$ there exists $\epsilon>0$ such that for any $(X,\lambda) \in P^1\mathcal{M}_{g,n}$ exactly one of the following holds:
	\begin{enumerate}
		\item $\displaystyle
		\liminf_{T \to \infty} \frac{\text{Leb}(\{t \in [0,T] \ | \ \text{tw}^t(X,\lambda) \in P^1K_\epsilon\})}{T} > 1 - \delta
		$. \\[-1.5pt]
		\item There is a simple closed curve $\gamma$ on $S_{g,n}$ with $i(\lambda, \gamma) = 0$ and $\ell_{\gamma}(X) < \epsilon$.
	\end{enumerate}
\end{theorem}
$ $

If condition (2) in Theorem \ref{theo:mw_quant_rec} holds, then the function 
\[
t \in \mathbf{R}_{\geq 0} \mapsto \ell_{\gamma}(\text{tw}^t(X,\lambda))
\] 
is constant and $< \epsilon$. In particular, $\text{tw}^t(X,\lambda) \notin P^1 K_\epsilon$ for every $t \geq 0$. Theorem \ref{theo:mw_quant_rec} precisely says that this is the only obstruction that could prevent the strong non-divergence property of the earthquake flow described by condition (1) from holding.\\

\textit{Sketch of the proof of Theorem \ref{theo:horoball_equid}.} Let $\gamma := (\gamma_1,\dots,\gamma_k)$ with $1 \leq k \leq 3g-3+n$ be an ordered simple closed multi-curve on $S_{g,n}$, $\mathbf{a}:=(a_1,\dots,a_k) \in (\mathbf{Q}_{>0})^k$ be a vector of positive rational weights on the components of $\gamma$, and $f \colon (\mathbf{R}_{\geq 0})^k \to \mathbf{R}_{\geq0}$ be a bounded, compactly supported, Borel measurable function with non-negative values and which is not almost everywhere zero with respect to the Lebesgue measure class. The following three propositions, describing the limit points of the sequence of probability measures $\{\widehat{\nu}_{\gamma,\mathbf{a}}^{f,L}/m_\gamma^{f,L}\}_{L>0}$ on $P^1\mathcal{M}_{g,n}$, together with Theorem \ref{theo:earth_erg} lead to a proof of Theorem \ref{theo:horoball_equid}.\\

\begin{proposition}
	\label{prop:earth_inv_0}
	Any weak-$\star$ limit point of the sequence of probability measures $\{\widehat{\nu}_{\gamma,\mathbf{a}}^{f,L}/m_\gamma^{f,L}\}_{L > 0}$ on $P^1\mathcal{M}_{g,n}$ is earthquake flow invariant.
\end{proposition}
$ $

\begin{proposition}
	\label{prop:hor_meas_ac_0}
	Any weak-$\star$ limit point of the sequence of probability measures $\{\widehat{\nu}_{\gamma,\mathbf{a}}^{f,L}/m_\gamma^{f,L}\}_{L > 0}$ on $P^1\mathcal{M}_{g,n}$ is absolutely continuous with respect to $\widehat{\nu}_{\text{Mir}}$.
\end{proposition}
$ $

\begin{proposition}
	\label{prop:hor_meas_nem_0}
	Any weak-$\star$ limit point of the sequence of probability measures $\{\widehat{\nu}_{\gamma,\mathbf{a}}^{f,L}/m_\gamma^{f,L}\}_{L > 0}$ on $P^1\mathcal{M}_{g,n}$ is a probability measure.
\end{proposition}
$ $

Proposition \ref{prop:earth_inv_0} follows from the fact that the Weil-Petersson volume form $\omega_{\text{wp}}$ on $\mathcal{T}_{g,n}$ is invariant under twist deformations; see (\ref{eq:hamilton}). Proposition \ref{prop:hor_meas_ac_0} follows from a series of estimates which rely in large part on the interpretation of the Thurston measure $\mu_\text{Thu}$ on $\mathcal{ML}_{g,n}$ as a limit of integral multi-curve counting measures; see (\ref{ML_counting_measure}). Proposition \ref{prop:hor_meas_nem_0} follows from Theorem \ref{theo:mw_quant_rec} and Mirzakhani's integration formulas in \cite{Mir07a}. See \S3 for detailed proofs of these propositions.\\

Let us prove Theorem \ref{theo:horoball_equid} using Propositions \ref{prop:earth_inv_0}, \ref{prop:hor_meas_ac_0}, and \ref{prop:hor_meas_nem_0}, and Theorem \ref{theo:earth_erg}.\\

\begin{proof}[Proof of Theorem \ref{theo:horoball_equid}]
	For the rest of this proof, convergence of measures on $P^1\mathcal{M}_{g,n}$ will always be considered with respect to the weak-$\star$ topology. To prove
	\[
	\lim_{L \to \infty} \frac{\widehat{\nu}_{\gamma,\mathbf{a}}^{f,L}}{m_\gamma^{f,L}} = \frac{\widehat{\nu}_{\text{Mir}}}{b_{g,n}}
	\]
	it is enough to show that every subsequence of $\{\widehat{\nu}_{\gamma,\mathbf{a}}^{f,L}/m_\gamma^{f,L}\}_{L>0}$ has a subsequence converging to $\widehat{\nu}_{\text{Mir}}/b_{g,n}$. Let $\{L_j\}_{j \in \mathbf{N}}$ be an arbitrary increasing sequence of positive real numbers converging to $+\infty$. By the Banach-Alouglu theorem, the space of measures of total mass $\leq 1$ on $P^1\mathcal{M}_{g,n}$ is compact with respect to the weak-$\star$ topology. In particular, we can find a subsequence $\{L_{j_k}\}_{k \in \mathbf{N}}$ of $\{L_j\}_{j \in \mathbf{N}}$ such that 
	\[
	\lim_{k \to \infty} \frac{\widehat{\nu}_{\gamma,\mathbf{a}}^{f,L_{j_k}}}{m_\gamma^{f,L_{j_k}}} = \widehat{\nu}_{\gamma,\mathbf{a}}^f
	\]
	for some measure $\widehat{\nu}_{\gamma,\mathbf{a}}^f$ on $P^1\mathcal{M}_{g,n}$ of total mass $\leq 1$. By Proposition \ref{prop:earth_inv_0}, $\widehat{\nu}_{\gamma,\mathbf{a}}^f$ is earthquake flow invariant. By Proposition \ref{prop:hor_meas_ac_0}, $\widehat{\nu}_{\gamma,\mathbf{a}}^f$ is absolutely continuous with respect to $\widehat{\nu}_{\text{Mir}}$. It follows from Theorem \ref{theo:earth_erg} that 
	\[
	\widehat{\nu}_{\gamma,\mathbf{a}}^f = c \cdot \widehat{\nu}_{\text{Mir}}
	\]
	for some constant $0 \leq c \leq 1/b_{g,n}$. By Proposition \ref{prop:hor_meas_nem_0}, it must be the case that $c = 1/b_{g,n}$. Putting things together we deduce
	\[
	\lim_{k \to \infty} \frac{\widehat{\nu}_{\gamma,\mathbf{a}}^{f,L_{j_k}}}{m_\gamma^{f,L_{j_k}}} =  \frac{\widehat{\nu}_{\text{Mir}}}{b_{g,n}}.
	\]
	As the sequence $\{L_j\}_{j \in \mathbf{N}}$ was arbitrary, this finishes the proof.
\end{proof}
$ $

\textit{Equidistribution of twist tori.} A natural follow-up question to Corollary \ref{cor:horosphere_equid} is whether similar equidistribution results hold for higher codimension analogues of horospheres. We now describe a particularly interesting instance of this question which remains open.\\

Fix an ordered pair of pants decomposition $\mathcal{P}:=(\gamma_1,\dots,\gamma_{3g-3+n})$ of $S_{g,n}$. Let $(\ell_i,\tau_i)_{i=1}^{3g-3+n} \in (\mathbf{R}_{>0} \times \mathbf{R})_{i=1}^{3g-3+n}$ be an arbitrary set of Fenchel-Nielsen coordinates of $\mathcal{T}_{g,n}$ adapted to $\mathcal{P}$. For every $L > 0$ consider the \textit{twist torus} $T_\mathcal{P}^L \subseteq \mathcal{T}_{g,n}$ given by
\[
T_\mathcal{P}^L := \{X \in \mathcal{T}_{g,n} \ | \ \ell_{\gamma_i}(X) = L, \ \forall i =1,\dots,3g-3+n\}.
\]
Every such twist torus supports a natural \textit{twist torus measure} $\tau_{\mathcal{P}}^L $ defined as
\[
d\tau_{\mathcal{P}}^L := d\tau_1 \cdots d\tau_{3g-3+n}.
\]
This definition is independent of the choice of Fenchel-Nielsen coordinates adapted to $\mathcal{P}$. We also denote by $\tau_{\mathcal{P}}^L$ the extension by zero of this measure to $\mathcal{T}_{g,n}$. This measure is $\text{Stab}(\mathcal{P})$-invariant. Let $\widetilde{\tau}_\mathcal{P}^L$ be the local pushforward of $\tau_{\mathcal{P}}^L$ to $\mathcal{T}_{g,n}/\text{Stab}(\mathcal{P})$ and $\widehat{\tau}_\mathcal{P}^L$ be the pushforward of $\widetilde{\tau}_\mathcal{P}^L$ to $\mathcal{M}_{g,n}$. One can check that the measures $\widehat{\tau}_\mathcal{P}^L$  on  $\mathcal{M}_{g,n}$ are finite. Denote by $t_\mathcal{P}^L$ the total mass of $\widehat{\tau}_\mathcal{P}^L$, i.e., $t_\mathcal{P}^L := \widehat{\tau}_\mathcal{P}^L(\mathcal{M}_{g,n}) < +\infty$.\\

\begin{question}
	Is it true that
	\[
	\lim_{L \to \infty} \frac{\widehat{\tau}_\mathcal{P}^L}{t_\mathcal{P}^L} = \frac{B(X) \cdot d\widehat{\mu}_{\text{wp}}(X)}{b_{g,n}}
	\]
	in the weak-$\star$ topology for measures on $\mathcal{M}_{g,n}$?
\end{question}
$ $

For other interesting open questions related to Mirzakhani's work see \cite{Wri19}.\\

\textit{Organization of the paper.} In \S 2 we present the background material necessary to understand the proofs of the main results. In \S 3 we present the proofs of Propositions \ref{prop:earth_inv_0}, \ref{prop:hor_meas_ac_0}, and \ref{prop:hor_meas_nem_0}, thus completing the proof of Theorem \ref{theo:horoball_equid}. In \S4 we give a brief overview of the proof of Theorem \ref{theo:horosphere_equid} focusing our attention on the proof of Proposition \ref{prop:mir_bd_og_0}. \\

\textit{Acknowledgments.} The author is very grateful to Alex Wright and Steven Kerckhoff for their invaluable advice, patience, and encouragement. The author would also like to thank Dat Pham Nguyen for very enlightening conversations.\\

\section{Background material}

$ $

\textit{Twist deformations and hyperbolic length functions.} Given a marked hyperbolic structure $X \in \mathcal{T}_{g,n}$, a measured geodesic lamination $\lambda \in \mathcal{ML}_{g,n}$, and a real number $t \in \mathbf{R}$, we denote by $\text{tw}_\lambda^t(X) \in \mathcal{T}_{g,n}$ the time $t$ \textit{twist deformation} (or \textit{earthquake deformation}) of $X$ along $\lambda$; see \S 2 in \cite{Ker83} for a definition. \\

By work of Wolpert, see \cite{Wol83}, twist deformations are the hamiltonian flows of hyperbolic length functions with respect to the Weil-Petersson symplectic form $\omega_{\text{wp}}$ on $\mathcal{T}_{g,n}$. More precisely, given an arbitrary measured geodesic lamination $\lambda \in \mathcal{ML}_{g,n}$, if $E_\lambda \colon \mathcal{T}_{g,n} \to T \mathcal{T}_{g,n}$ denotes the vector field induced by twist deformations along $\lambda$ and $\ell_\lambda \colon \mathcal{T}_{g,n} \to \mathbf{R}_{>0}$ denotes the hyperbolic length function of $\lambda$, then
\begin{equation}
\label{eq:hamilton}
\omega_{\text{wp}}(E_\lambda,\cdot) = -d\ell_\lambda.
\end{equation}
In particular, the Weil-Petersson symplectic form $\omega_{\text{wp}}$, the Weil-Petersson volume form $v_\text{wp}$, and the Weil-Petersson measure $\mu_\text{wp}$ on $\mathcal{T}_{g,n}$ are invariant under twist deformations.\\

By work of Kerchkoff, see \cite{Ker85}, hyperbolic length functions of measured geodesic laminations are a continuous family of smooth functions on $\mathcal{T}_{g,n}$.\\

\begin{theorem}
	\label{theo:smooth_length}
	The function $\ell \colon \mathcal{ML}_{g,n} \to \mathcal{C}^\infty(\mathcal{T}_{g,n})$ given by $\lambda \mapsto \ell_{\lambda}$ is continuous with respect to the $\mathcal{C}^\infty$-topology for functions on compact subsets of $\mathcal{T}_{g,n}$. 
\end{theorem}
$ $

\textit{The earthquake flow.}  The \textit{earthquake flow} on the bundle $P^1 \mathcal{T}_{g,n}$ of unit length measured geodesic laminations over Teichmüller space
\[
\{\text{tw}^t \colon P^1\mathcal{T}_{g,n} \to P^1\mathcal{T}_{g,n}\}_{t \in \mathbf{R}}
\]
 is given by
\[
\text{tw}^t(X,\lambda) := (\text{tw}^t_\lambda(X),\lambda)
\]
for every $t \in \mathbf{R}$ and every $(X,\lambda) \in P^1 \mathcal{T}_{g,n}$. As the mapping class group action on $P^1 \mathcal{T}_{g,n}$ commutes with the earthquake flow, the bundle $P^1 \mathcal{M}_{g,n} = P^1 \mathcal{T}_{g,n}/\text{Mod}_{g,n}$ of unit length measured geodesic laminations over moduli space also carries an \textit{earthquake flow}
\[
\{\text{tw}^t \colon P^1\mathcal{M}_{g,n} \allowbreak \to P^1\mathcal{M}_{g,n}\}_{t \in \mathbf{R}}.
\]
$ $

\textit{The Thurston measure.}  Let $\mathcal{ML}_{g,n}(\mathbf{Z}) \subseteq \mathcal{ML}_{g,n}$ be the subset of all unordered, positively integrally weighted, simple closed multi-curves on $S_{g,n}$. For every $L > 0$ consider the counting measure $\mu^L$ on $\mathcal{ML}_{g,n}$ given by 
\begin{equation}
\label{ML_counting_measure}
\mu^L := \frac{1}{L^{6g-6+2n}} \sum_{\gamma \in \mathcal{ML}_{g,n}(\mathbf{Z})} \delta_{ \frac{1}{L} \cdot \gamma}.
\end{equation}
Using train track coordinates, see \cite{PH92}, one can check that, as $L \to \infty$, this sequence of counting measures converges, in the weak-$\star$ topology, to a non-zero, locally finite, $\text{Mod}_{g,n}$-invariant, Lebesgue class measure $\mu_{\text{Thu}}$ on $\mathcal{ML}_{g,n}$. We refer to this measure as the \textit{Thurston measure} of $\mathcal{ML}_{g,n}$. \\

\textit{The Mirzakhani measure.} For every marked hyperbolic structure $X \in \mathcal{T}_{g,n}$, consider the measure $\mu_{\text{Thu}}^X$ on the fiber  $P^1_X\mathcal{T}_{g,n}$ of the bundle $P^1\mathcal{T}_{g,n}$ above $X$, which to every Borel measurable subset $A \subseteq P^1_X\mathcal{T}_{g,n}$ assigns the value
\[
\mu_{\text{Thu}}^X(A) := \mu_{\text{Thu}}([0,1] \cdot A).
\]
On the bundle $P^1\mathcal{T}_{g,n}$ one obtains a natural measure $\nu_{\text{Mir}}$, called the \textit{Mirzakhani measure}, by considering the following disintegration formula
\[
d\nu_{\text{Mir}}(X,\lambda) := d\mu_{\text{Thu}}^X(\lambda) \  d\mu_{\text{wp}}(X).
\]
In \cite{Mir08a}, Mirzakhani showed that the measure $\nu_{\text{Mir}}$ is invariant under both the earthquake flow and the action of the mapping class group. We denote by $\widehat{\nu}_{\text{Mir}}$ the local pushforward to $P^1\mathcal{M}_{g,n}$ of the measure $\nu_{\text{Mir}}$ on $P^1\mathcal{T}_{g,n}$ under the quotient map $P^1\mathcal{T}_{g,n} \to P^1\mathcal{M}_{g,n}$; we also refer to this measure as the \textit{Mirzakhani measure}. This measure is also invariant under the earthquake flow and its pushforward under the bundle map $P^1\mathcal{M}_{g,n} \to \mathcal{M}_{g,n}$ is given by
\[
B(X) \ d\widehat{\mu}_\text{wp}(X),
\]
where $B \colon \mathcal{M}_{g,n} \to \mathbf{R}_{>0}$ is the \textit{Mirzakhani function} defined in (\ref{eq:mir_fn}). As mentioned in (\ref{eq:b_gn}), the measure $\widehat{\nu}_{\text{Mir}}$ on $P^1\mathcal{M}_{g,n}$ is finite.\\

\textit{Volume forms induced on level sets.} Let $M$ be a (not necessarily compact) real smooth manifold, $\omega$ be a smooth volume form on $M$, and $f \colon M \to \mathbf{R}$ be a smooth function with no critical points. By the regular value theorem, the level sets $f^{-1}(L) \subseteq M$ are smoothly embedded codimension 1 real submanifolds of $M$ for every $L \in \text{Im}(f)$.\\

Fix $L \in \text{Im}(f)$ and consider the smoothly embedded codimension 1 real submanifold $f^{-1}(L) \subseteq M$. The volume form $\omega$ on $M$ induces a volume form $\omega_f^L$ on $f^{-1}(L)$ by contraction. More precisely, $\omega_f^L$ is the volume form obtained by restricting to $f^{-1}(L)$ the contraction of $\omega$ by any vector field $X$ on $M$ satisfying $df(X) \equiv 1$. Such a vector field always exists: given any Riemannian metric $\langle \cdot,\cdot \rangle$ on $M$, the vector field $X:=\nabla f/ \langle \nabla f, \nabla f \rangle$ satisfies the desired condition. The definition of $\omega_f^L$ is a pointwise definition: for every point $x \in f^{-1}(L)$, the volume form $\omega_f^L$ at $x$ is completely determined by the volume form $\omega$ and the vector field $X$ at $x$.\\

The volume form $\omega_f^L$ on $f^{-1}(L)$ is well defined, i.e., it is independent of the choice of vector field $X$ on $M$ satisfying $df(X)\equiv 1$. Indeed, let $X_1,X_2$ be a pair of vector fields on $M$ satisfying $df(X_1) \equiv  df(X_2) \equiv 1$. It follows that $X_1-X_2 \in \ker(df)$. In particular, $(X_1 - X_2)_x \in T_x f^{-1}(L)$ for every $x \in f^{-1}(L)$. Contracting $\omega$ by $X_1-X_2$ and restricting to $f^{-1}(L)$ yields the zero volume form because volume forms are always zero when evaluated on linearly dependent vectors. We deduce that the volume forms $\omega_f^L$ obtained from $X_1$ and $X_2$ are the same.\\

Let $\mu:= |\omega|$ be the measure induced by the volume form $\omega$ on $M$. For every $L \in \text{Im}(f)$, let $\mu_f^L := |\omega_f^L|$ be the measure induced by the volume form $\omega_f^L$ on $f^{-1}(L)$. We also denote by $\mu_f^L$ the extension by zero of this measure to $M$. The measure $\mu$ can be disintegrated into the measures $\eta_f^L$ in the following way:\\

\begin{proposition}
	\label{prop:meas_disint}
	The following equality of measures on $M$ holds:
	\[
	\mu = \int_{\mathbf{R}} \mu_f^r \ dr.
	\]
\end{proposition}
$ $

By a \textit{flow datum} of the function $f$ on $M$ we mean a triple $(K,K',X)$, where $K \subseteq K' \subseteq M$ are compact subsets of $M$ and $X$ is a vector field on $M$ such that $df(X) \equiv 1$ on $K$ and which vanishes outside of $K'$. Let $\{\varphi_t \colon M \to M\}_{t \in \mathbf{R}}$ be the one-parameter group of diffeomorphisms induced by $X$ on $M$ and $\varphi \colon M \times \mathbf{R} \to M$ be the smooth map given by $\varphi(x,t) := \varphi_t(x)$ for every $x \in M$ and every $t \in \mathbf{R}$. The diffeomorphisms $\varphi_t \colon M \to M$ induce nowhere vanishing determinant functions $\text{det}(\varphi_t) \colon M \to \mathbf{R}_{>0}$ which to every $x \in M$ assign the value 
\[
\text{det}(\varphi_t)(x)  := \frac{(\varphi_t^* \omega)_x}{\omega_x}.
\]
$ $

A triple $(L,V,I)$ with $L \in \text{Im}(f)$, $V \subseteq K \cap f^{-1}(L)$, and $I \subseteq \mathbf{R}$ an interval containing $0$ is said to have \textit{constant speed} with respect to the flow datum $(K,K',X)$ if $\varphi(V \times I) \subseteq K$. Notice that $\varphi_r(V) \subseteq f^{-1}(L+r)$ for every $r \in I$. In this setting, the variation of the volume forms $\omega_f^r$ with respect to the flow $\{\varphi_t \colon M \to M\}_{t \in \mathbf{R}}$ admits the following description. \\

\begin{proposition}
	\label{prop:vol_var}
	Let $(K,K',X)$ be a flow datum of $f$ on $M$ and $(L,V,I)$ be a constant speed triple with respect to this datum. Then, for every $r \in I$, the following equality of volume forms on $V$ holds: 
	\[
	\varphi_r^* \omega_f^{L+r}= \text{det}(\varphi_r) \cdot \omega_f^L.
	\]
\end{proposition}
$ $

\textit{Horospherical measures on Teichmüller space.} Let $\lambda \in \mathcal{ML}_{g,n}$ be a measured geodesic laminations on $S_{g,n}$. In the context of the previous discussion consider
\begin{align*}
M &:= \mathcal{T}_{g,n},\\
\omega &:= v_{\text{wp}}, \\
f &:= \ell_{\lambda} \colon \mathcal{T}_{g,n} \to \mathbf{R}_{>0}.
\end{align*}
By Theorem \ref{theo:smooth_length}, the function $\ell_{\lambda}$ is smooth. The duality described in (\ref{eq:hamilton}) and the fact that the Weil-Petersson symplectic form is non-degenerate ensure that $\ell_{\lambda}$ has no critical points. The function $\ell_{\lambda} \colon \mathcal{T}_{g,n} \to \mathbf{R}_{>0}$ is also surjective; see \cite{Thu98} for an argument using stretch paths. It follows from the previous discussion that, given $L > 0$, one can consider a \textit{horospherical volume form} and a \textit{horospherical measure}
\[
\omega_\lambda^L := \omega_{\ell_\lambda}^L \quad , \quad
\eta_\lambda^L := \mu_{\ell_\lambda}^L
\]
supported on the \textit{horosphere} $S_\lambda^L  \subseteq \mathcal{T}_{g,n}$ given by
\[
S_\lambda^L := \ell_\lambda^{-1}(L).
\]
As the function $\ell_\lambda$ is invariant under twists along $\lambda$ and as the Weil-Petersson volume form $v_{\text{wp}}$ is also invariant under such twists, see  (\ref{eq:hamilton}), the volume form $\omega_\lambda^L$ and the measure $\eta_\lambda^L$ on $\mathcal{T}_{g,n}$ are invariant under such twists as well.\\

Let $\gamma := (\gamma_1,\dots,\gamma_k)$ with $1 \leq k \leq 3g-3+n$ be an ordered simple closed multi-curve on $S_{g,n}$, $\mathbf{a}:= (a_1,\dots,a_k) \in (\mathbf{Q}_{>0})^k$ be a vector of positive rational weights on the components of $\gamma$, and $\mathbf{a} \cdot \gamma \in \mathcal{ML}_{g,n}(\mathbf{Q})$ be as in (\ref{eq:weight_curve}). As a particular example of the construction above, we recover the definitions introduced in \S 1: 
\[
S_{\mathbf{a}\cdot\gamma}^L = S_{\gamma,\mathbf{a}}^L \quad , \quad \omega_{\mathbf{a}\cdot \gamma}^L = \omega_{\gamma,\mathbf{a}}^L \quad , \quad \eta_{\mathbf{a}\cdot \gamma}^L = \eta_{\gamma,\mathbf{a}}^L.
\]
$ $

\textit{The symmetric Thurston metric.} Consider the \textit{asymmetric Thurston metric} $d_\text{Thu}'$ on $\mathcal{T}_{g,n}$ which to every pair of marked hyperbolic structures $X,Y \in \mathcal{T}_{g,n}$ assigns the distance
\[
d_\text{Thu}'(X,Y) := \sup_{\lambda \in \mathcal{ML}_{g,n}} \log \left(\frac{\ell_\lambda(Y)}{\ell_\lambda(X)} \right).
\]
This metric is proper, i.e., sets of bounded diameter with respect it are compact, Finsler, i.e., it is induced by a continuous family of fiberwise asymmetric norms $\|\cdot\|$ on $T\mathcal{T}_{g,n}$, and induces the usual topology on $\mathcal{T}_{g,n}$. Given the asymmetric nature of this metric, it is convenient to consider the \textit{symmetric Thurston metric} $d_\text{Thu}$ which to every pair of marked hyperbolic structures $X,Y \in \mathcal{T}_{g,n}$ assigns the distance
\[
d_\text{Thu}(X,Y) := \max\{d_\text{Thu}'(X,Y), d_\text{Thu}'(Y,X)\}.
\]
A pair of marked hyperbolic structures $X,Y \in \mathcal{T}_{g,n}$ satisfy $d_\text{Thu}(X,Y) \leq \epsilon$ if and only if
\[
e^{-\epsilon} \ell_\lambda(X) \leq \ell_\lambda(Y) \leq e^\epsilon  \ell_{\lambda}(X), \ \forall \lambda \in \mathcal{ML}_{g,n}.
\]
The symmetric Thurston metric $d_\text{Thu}$ is proper and induces the usual topology on $\mathcal{T}_{g,n}$. Denote by $U_X(\epsilon) \subseteq \mathcal{T}_{g,n}$ the open ball of radius $\epsilon > 0$ centered at $X \in \mathcal{T}_{g,n}$ with respect to $d_\text{Thu}$. For more details on the theory of the asymmetric and symmetric Thurston metrics, see \cite{Thu98} and \cite{Pap15}. \\

As a consequence of the fact that the symmetric Thurston metric $d_\text{Thu}$ on $\mathcal{T}_{g,n}$ is the maximum of two Finsler metrics, one can deduce the following estimates using standard compactness arguments.\\

\begin{lemma}
	\label{lemma:len_var}
	Let $\{\varphi_t \colon \mathcal{T}_{g,n} \to \mathcal{T}_{g,n}\}_{t \in \mathbf{R}}$ be a one-parameter group of diffeomorphisms of $\mathcal{T}_{g,n}$ induced by a smooth vector field $F$ supported on a compact subset $K \subseteq \mathcal{T}_{g,n}$. Let
	\[
	\|F\|_K := \sup_{Y \in K} \|\pm F_Y\|_Y < + \infty.
	\]
	Then, for every $X \in K$ and every $t \in \mathbf{R}$,
	\[
	d_\text{Thu}(X,\varphi_t(X)) \leq  \|F\|_K \cdot t.
	\]
\end{lemma}
$ $

\begin{lemma}
	\label{lemma:thu_ball_bound}
	Let $K \subseteq \mathcal{T}_{g,n}$ be a compact subset. There exist constants $C >0$ and $\epsilon_0 >0$ such that for every $X \in K$ and every $0 < \epsilon < \epsilon_0$, 
	\[
	C^{-1} \cdot \epsilon^{6g-6+2n} \leq \mu_{wp}(U_X(\epsilon)) \leq C \cdot \epsilon^{6g-6+2n}.
	\]
\end{lemma}
$ $

\textit{Teichmüller and moduli spaces of hyperbolic surfaces with geodesic boundary.} Let $g',n',b' \in \mathbf{Z}_{\geq 0}$ be a triple of non-negative integers satisfying $2 - 2g' - n' - b' < 0$. Consider a fixed connected, oriented surface $S_{g',n'}^{b'}$ of genus $g'$ with $n'$ punctures and $b'$ labeled boundary components $\beta_1,\dots,\beta_{b'}$ and a vector $\mathbf{L}:= (L_i)_{i=1}^{b'} \in (\mathbf{R}_{>0})^{b'}$ of positive real numbers. \\

We denote by $\mathcal{T}_{g',n'}^{b'}(\mathbf{L})$ the \textit{Teichmüller space} of  marked, oriented, complete, finite area hyperbolic structures on $S_{g',n'}^{b'}$ with labeled geodesic boundary components whose lengths are given by $\mathbf{L}$. The \textit{mapping class group} of $S_{g',n'}^{b'}$, denoted $\text{Mod}_{g',n'}^{b'}$, is the group of isotopy classes of orientation preserving diffeomorphisms of $S_{g',n'}^{b'}$ that fix each puncture and set-wise fix each boundary component. The quotient $\mathcal{M}_{g',n'}^{b'}(\mathbf{L}) := \mathcal{T}_{g',n'}^{b'}(\mathbf{L}) / \text{Mod}_{g',n'}^{b'}$ is the \textit{moduli space} of oriented, complete, finite area hyperbolic structures on $S_{g',n'}^{b'}$ with labeled geodesic boundary components whose lengths are given by $\mathbf{L}$. \\

As a consequence of \textit{Wolpert's magic formula}, see \cite{Wol85}, the Weil-Petersson volume form $v_{\text{wp}}$ on $\mathcal{T}_{g',n'}^{b'}(\mathbf{L})$ can be expressed in any set of Fenchel-Nielsen coordinates  $(\ell_i,\tau_i)_{i=1}^{3g'-3+n'+b'} \in (\mathbf{R}_{>0} \times \mathbf{R})^{3g'-3+n'+b'}$ as
\[
v_{\text{wp}} = \prod_{i=1}^{3g'-3+n'+b'} d\ell_i \wedge d\tau_i.
\]
$ $

Of particular importance for us will be the \textit{total Weil-Petersson volume} of the moduli space $\mathcal{M}_{g',n'}^{b'}(\mathbf{L})$, which we denote by
\[
V_{g',n'}^{b'}(\mathbf{L}) := \text{Vol}_{\text{wp}}(\mathcal{M}_{g',n'}^{b'}(\mathbf{L})).
\]
The following remarkable theorem due to Mirzakhani, see \cite{Mir07a} and \cite{Mir07c}, shows that $V_{g',n'}^{b'}(\mathbf{L})$ behaves like a polynomial on the $\mathbf{L}$ variables.\\

\begin{theorem}
	\label{theo:vol_pol}
	The total Weil-Petersson volume
	\[
	V_{g',n'}^{b'}(L_1,\dots,L_{b'})
	\]
	is a polynomial of degree $3g'-3+n'+b'$ on the variables $L_1^2,\dots,L_{b'}^2$. Moreover, if we denote
	\[
	V_{g',n'}^{b'}(L_1,\dots,L_{b'}) = \sum_{\substack{\alpha \in (\mathbf{Z}_{\geq 0})^{b'}, \\ |\alpha| \leq 3g'-3+n'+b'} } c_\alpha \cdot L_1^{2\alpha_1} \cdots L_{b'}^{2\alpha_{b'}},
	\]
	where $|\alpha| := \alpha_1 + \cdots + \alpha_{b'}$ for every $\alpha \in (\mathbf{Z}_{\geq 0})^{b'}$, then $c_\alpha \in \mathbf{Q}_{>0} \cdot \pi^{6g'-6+2n' +2b' - 2|\alpha|}$. In particular, the leading coefficients of $V_{g',n'}^{b'}(L_1,\dots,L_{b'})$ are all in $\mathbf{Q}_{> 0}$.
\end{theorem}
$ $

\begin{remark}
	If the surface $S_{g',n'}^{b'}$ is a pair of pants, i.e., if $g' = 0$ and $n'+b' = 3$, then, for any $\mathbf{L} := (L_i)_{i=1}^{b'} \in (\mathbf{R}_{> 0})^{b'}$, the moduli space $\mathcal{M}_{g',n'}^{b'}(\mathbf{L})$ has exactly one point. We will adopt the convention
	\[
	V_{g',n'}^{b'}(\mathbf{L}) := 1.
	\]
\end{remark}
$ $

We will also make use of the following version of Bers's theorem for surfaces with boundary; see Theorem 12.8 in \cite{FM11} for a proof. Denote
\[
M(\mathbf{L}) := \max_{i=1,\dots,b'} L_i.
\]
$ $

\begin{theorem}
	\label{theo:bers}
	There exist constants $A_{g',n'}^{b'}, B_{g',n'}^{b'} > 0$ such that for every $\mathbf{L}:= (L_i)_{i=1}^{b'} \in (\mathbf{R}_{>0})^{b'}$, every 
	$X \in \mathcal{M}_{g',n'}^{b'}(\mathbf{L})$, and every simple closed geodesic $\alpha_1$ on $X$, there exists a geodesic pair of pants decomposition $\{\alpha_j\}_{j=1}^{3g'-3+b'+n'} $of $X$ containing $\alpha_1$ such that
	\[
	\ell_{\alpha_j}(X) \leq A_{g',n'}^{b'} + B_{g',n'}^{b'} \cdot \max\{ M(\mathbf{L}),\ell_{\alpha_1}(X)\}, \ \forall j = 1,\dots,3g'-3+n'+b'.
	\]
\end{theorem}
$ $

\textit{The cut and glue fibration.} Given a simple closed curve $\alpha$ on $S_{g,n}$, let
\[
\text{Stab}_0(\alpha) \subseteq \text{Mod}_{g,n}
\]
be the subgroup of all mapping classes of $S_{g,n}$ that fix $\alpha$ (up to isotopy) together with its orientation (although $\alpha$ is unoriented, it admits two possible orientations; what is being asked is that the mapping class sends each such orientation back to itself). More generally, given an ordered simple closed multi-curve $\gamma := (\gamma_1,\dots,\gamma_k)$ on $S_{g,n}$ with $1 \leq k \leq 3g-3+n$, let
\[
\text{Stab}_0(\gamma) := \bigcap_{i=1}^k \text{Stab}_0(\gamma_i) \subseteq \text{Mod}_{g,n}
\]
be the subgroup of all mapping classes of $S_{g,n}$ that fix each component of $\gamma$ (up to isotopy) together with their respective orientations.\\

For the rest of this discussion fix an ordered simple closed multi-curve $\gamma := (\gamma_1,\dots,\gamma_k)$ on $S_{g,n}$ with $1 \leq k \leq 3g-3+n$. The quotient space
\[
\mathcal{T}_{g,n}/\text{Stab}_0(\gamma)
\]
fibers in a natural way over a product of moduli spaces of surfaces with boundary of less complexity than $S_{g,n}$. This fibration, which we refer to as the \textit{cut and glue fibration} of $\mathcal{T}_{g,n}/\text{Stab}_0(\gamma)$, is particularly useful for computing integrals over $\mathcal{M}_{g,n}$ with respect to the Weil-Petersson volume form. We now describe such fibration.\\

Let $S_{g,n}(\gamma)$ be the (potentially disconnected) oriented topological surface with boundary obtained by cutting $S_{g,n}$ along the components of $\gamma$. Let $c \in \mathbf{Z}_{>0}$ be the number of components of such surface. Fixing an orientation on each component of $\gamma$, we can keep track of which components of $S_{g,n}(\gamma)$ lie to the left and to the right of each component of $\gamma$, so we can label the components of $S_{g,n}(\gamma)$ in a consistent way, say $\Sigma_j$ with $j \in \{1,\dots,c\}$. As the components of $\gamma$ are ordered, this induces a labeling of the boundary components of each $\Sigma_j$. Let $g_j,n_j,b_j \in \mathbf{Z}_{\geq 0}$ with $2 - 2g_j - n_j - b_j < 0$ be the triple of non-negative integers such that $\Sigma_j$ is homeomorphic to $S_{g_j,n_j}^{b_j}$. Fix a homeomorphism between these surfaces respecting the labeling of their boundary components. \\

We consider the space $\Omega_{g,n}(\gamma)$ given by all pairs
\[
((\ell_i)_{i=1}^k, (X_j)_{j=1}^c),
\]
with
\begin{align*}
(\ell_i)_{i=1}^k &\in (\mathbf{R}_{>0})^k,\\
(X_j)_{j=1}^c &\in \prod_{j=1}^c \mathcal{M}_{g_j,n_j}^{b_j}(\mathbf{L}_j),
\end{align*}
where $\mathbf{L}_j \in (\mathbf{R}_{>0})^{b_j}$ is defined using the vector $(\ell_i)_{i=1}^k \in (\mathbf{R}_{>0})^k$ and the correpondence between the labeling of the boundary components of $\Sigma_j$ and the order of the components of $\gamma$. \\

The quotient $\mathcal{T}_{g,n} /\text{Stab}_0(\gamma)$ can be identified with the space of all pairs $(X,\alpha)$ where $X \in \mathcal{M}_{g,n}$ and $\alpha := (\alpha_1,\dots,\alpha_k)$ is an ordered simple closed multi-geodesic on $X$ of the same topological type as $\gamma$ with a choice of orientation on each component, modulo the equivalence relation $(X',\alpha') \sim (X'',\alpha'')$ if and only if there exists an orientation-preserving isometry $I \colon X' \to X''$ sending the components of $\alpha'$ to the components of $\alpha''$ respecting their ordering and orientations. The \textit{cut and glue fibration} of $\mathcal{T}_{g,n} /\text{Stab}_0(\gamma)$ is the map $\Psi \colon \mathcal{T}_{g,n} /\text{Stab}_0(\gamma) \to \Omega_{g,n}(\gamma)$ which to every pair $(X,\alpha) \in \mathcal{T}_{g,n} /\text{Stab}_0(\gamma)$ assigns the pair
\[
((\ell_i)_{i=1}^k, (X_j)_{j=1}^c) \in \Omega_{g,n}(\gamma),
\]
with
\begin{align*}
(\ell_i)_{i=1}^k &:= (\ell_{\alpha_i}(X))_{i=1}^k,\\
(X_j)_{j=1}^c &:= (X(\alpha)_j)_{j=1}^c,
\end{align*}
where $X(\alpha)_j \in \mathcal{M}_{g_j,n_j}^{b_j}(\mathbf{L}_j)$ denotes the $j$-th component (according to the labeling and orientation of the components of $\alpha$) of the hyperbolic surface with geodesic boundary obtained by cutting $X$ along the components of $\alpha$. The fiber above any pair $((\ell_i)_{i=1}^k, (X_j)_{j=1}^c) \in \Omega_{g,n}(\gamma)$ is given by all possible ways of glueing the pieces $(X_j)_{j=1}^c$ respecting the labelings. Given any $(X,\alpha)$ in such a fiber, the whole fiber can be recovered by considering all possible twists of $X$ along the components of $\alpha$.\\

We make two additional important observations about the fibers of $\Psi$. First, given a fiber of $\Psi$ and a pair $(X,\alpha)$ in such fiber, even though the amount of twist of the hyperbolic surface $X$ along the components of $\alpha$ is not well defined, it is defined up to a choice of base point. In particular, there are well defined $1$-forms $d\tau_{\alpha_i}$ on such fiber measuring the infinitesimal twist along the components $\alpha_i$ of $\alpha$. Second, if we measure the size of the fibers using the volume form obtained by wedging all such $1$-forms $d\tau_{\alpha_i}$, then, generically (i.e, for every set of lengths and generically with respect to Weil-Petersson volume forms), the volume of the fiber of $\Psi$ above the pair  $((\ell_i)_{i=1}^k, (X_j)_{j=1}^c) \in \Omega_{g,n}(\gamma)$ is 
\[
2^{-\rho_{g,n}(\gamma)} \cdot \ell_1 \cdots \ell_k,
\]
where $\rho_{g,n}(\gamma)$ is the number of components of $\gamma$ that bound (on any of its sides) a torus with one boundary component; we will refer to such pairs $((\ell_i)_{i=1}^k, (X_j)_{j=1}^c) \in \Omega_{g,n}(\gamma)$ as \textit{generic pairs}. The factor $2^{-\rho_{g,n}(\gamma)}$ represents the fact that every hyperbolic tori with one geodesic boundary component has a non-trivial isometric involution preserving the boundary component.\\

It follows from Wolpert's magic formula that the Weil-Petersson volume form $v_\text{wp}$ on $\mathcal{T}_{g,n} /\text{Stab}_0(\gamma)$ can be written locally, up to sign and orbifold factors, in terms of the cut and glue fibration $\Psi \colon \mathcal{T}_{g,n} /\text{Stab}_0(\gamma) \to \Omega_{g,n}(\gamma)$ as
\[
v_{\text{wp}} = \left(\prod_{i=1}^k d\tau_{\alpha_i} \right) \wedge \left(\prod_{i=1}^c v_{\text{wp}}^j \right)  \wedge \left(\prod_{i=1}^k d\ell_i \right),
\]
where $v_\text{wp}^j$ denotes the Weil-Petersson volume form on $\mathcal{M}_{g_j,n_j}^{b_j}(\mathbf{L}_j)$. To take care of the orbifold factors we incorporate the \textit{automorphism discrepancy factor} $\sigma_{g,n}(\gamma) \in \mathbf{Q}_{>0}$ given by
\[
\sigma_{g,n}(\gamma) := \frac{\prod_{j=1}^c |K_{g_j,n_j}^{b_j}|}{|\text{Stab}_0(\gamma)\cap K_{g,n}|},
\]
where $K_{g_j,n_j}^{b_j} \triangleleft \text{Mod}_{g_j,n_j}^{b_j}$ is the kernel of the mapping class group action on $\mathcal{T}_{g_j,n_j}^{b_j}$ and $K_{g,n} \triangleleft \text{Mod}_{g,n}$ is the kernel of the mapping class group action on $\mathcal{T}_{g,n}$. Indeed, $|K_{g,n}^b|$ is generically (with respect to the Weil-Petersson volume form) the number of automorphisms of a point in the orbifold $\mathcal{M}_{g,n}^b(\mathbf{L})$ and analogously for $|K_{g,n}|$. For example, if $g =2$, $n=0$, and $\gamma$ is a separating simple closed curve on $S_{2,0}$, then $\sigma_{2,0}(\gamma) = 4/2 = 2$.\\

We record the discussion above in the following theorem; this is a reformulation of Mirzakhani's integration formulas in \cite{Mir07a}.\\

\begin{theorem}
	\label{theo:cut_glue_fib}
	In terms of the cut and glue fibration
	\[
	\Psi \colon \mathcal{T}_{g,n} /\text{Stab}_0(\gamma) \to \Omega_{g,n}(\gamma),
	\]
	the Weil-Petersson volume form $v_{\text{wp}}$ on $\mathcal{T}_{g,n} /\text{Stab}_0(\gamma)$ can be written locally up to sign as
	\[
	v_{\text{wp}} = \sigma_{g,n}(\gamma) \cdot \left(\prod_{i=1}^k d\tau_{\alpha_i} \right) \wedge \left(\prod_{i=1}^c v_{\text{wp}}^j \right)  \wedge \left(\prod_{i=1}^k d\ell_i \right),
	\]
	where $v_\text{wp}^j$ denotes the Weil-Petersson volume form on $\mathcal{M}_{g_j,n_j}^{b_j}(\mathbf{L}_j)$. Generically, the volume of the fiber of $\Psi$ above the pair  $((\ell_i)_{i=1}^k, (X_j)_{j=1}^c) \in \Omega_{g,n}(\gamma)$ is equal to
	\[
	2^{-\rho_{g,n}(\gamma)} \cdot \ell_1 \cdots \ell_k.
	\]
\end{theorem}
$ $

\section{Equidistribution of horoballs}

$ $

\textit{Setting.} For the rest of this section, fix an ordered simple closed multi-curve $\gamma := (\gamma_1,\dots,\gamma_k)$ on $S_{g,n}$ with $1 \leq k \leq 3g-3+n$, a vector $\mathbf{a}:=(a_1,\dots,a_k) \in (\mathbf{Q}_{>0})^k$ of positive rational weights on the components of $\gamma$, and a bounded, compactly supported, Borel measurable function $f \colon (\mathbf{R}_{\geq 0})^k \to \mathbf{R}_{\geq0}$ with non-negative values and which is not almost everywhere zero with respect to the Lebesgue measure class. The goal of this section is to prove Propositions \ref{prop:earth_inv_0}, \ref{prop:hor_meas_ac_0}, and \ref{prop:hor_meas_nem_0}, and thus complete the proof of Theorem \ref{theo:horoball_equid}. \\

\textit{Total mass.} Before proceeding any further, it is important to ascertain that the measures $\widehat{\mu}_\gamma^{f,L}$ on $\mathcal{M}_{g,n}$ and $\widehat{\nu}_{\gamma,\mathbf{a}}^{f,L}$ on $P^1\mathcal{M}_{g,n}$ are finite. We actually derive explicit formulas for the total mass $m_\gamma^{f,L}$ of these measures and study their asymptotics as $L \to \infty$. These formulas play an important role in the proof of Theorem \ref{theo:horoball_equid} as well as in the applications in \cite{AH19b}. \\

Following the notation of Theorem \ref{theo:cut_glue_fib}, given $\mathbf{L} := (\ell_i)_{i=1}^k \in (\mathbf{R}_{>0})^k$, let
\[
V_{g,n}(\gamma,\mathbf{L}) := \frac{1}{[\text{Stab}(\gamma):\text{Stab}_0(\gamma)]} \cdot \sigma_{g,n}(\gamma) \cdot 2^{-\rho_{g,n}{(\gamma)}} \cdot \prod_{j=1}^c V_{g_j,n_j}^{b_j}(\mathbf{L}_j) \cdot \ell_1,\dots,\ell_k.
\]
By Theorem \ref{theo:vol_pol}, $V_{g,n}(\gamma,\mathbf{L})$ is a polynomial of degree $6g-6+2n-k$ in the $\mathbf{L}$ variables with non-negative coefficients and rational leading coefficients. Denote by $W_{g,n}(\gamma,\mathbf{L})$ the polynomial obtained by adding up all the leading (maximal degree) monomials of $V_{g,n}(\gamma,\mathbf{L})$.\\

\begin{proposition}
	\label{prop:total_hor_meas}
	For every $L > 0$,
	\[
	m_\gamma^{f,L} = \int_{\mathbf{R}^k} f(\mathbf{L}) \cdot V_{g,n}(\gamma,L \cdot \mathbf{L}) \cdot L^k \  d \mathbf{L},
	\]
	where $d\mathbf{L} := d\ell_1 \cdots d\ell_k$. In particular,
	\[
	\lim_{L \to \infty} \frac{m_\gamma^{f,L}}{L^{6g-6+2n}} = \int_{\mathbf{R}^k} f(\mathbf{L}) \cdot W_{g,n}(\gamma,\mathbf{L}) \ d \mathbf{L}.
	\]
\end{proposition}
$ $

\begin{proof}
	Let $L > 0$ be arbitrary. Recall that by definition
	\[
	m_\gamma^{f,L} := \widehat{\mu}_\gamma^{f,L}(\mathcal{M}_{g,n}).
	\]
	As $\widehat{\mu}_\gamma^{f,L}$ is the pushforward to $\mathcal{M}_{g,n}$ of the measure $\widetilde{\mu}_\gamma^{f,L}$ on $\mathcal{T}_{g,n}/\text{Stab}(\gamma)$, 
	\[
	\widehat{\mu}_\gamma^{f,L}(\mathcal{M}_{g,n}) = \widetilde{\mu}_\gamma^{f,L}(\mathcal{T}_{g,n}/\text{Stab}(\gamma)).
	\]
	Let $\dot{\mu}_{\gamma}^{f,L}$ be the local pushforward to  $\mathcal{T}_{g,n}/\text{Stab}_0(\gamma)$ of the measure $\mu_{\gamma}^{f,L}$ on $\mathcal{T}_{g,n}$. As $[\text{Stab}(\gamma) : \text{Stab}_0(\gamma)] < \infty$,
	\[
	\widetilde{\mu}_\gamma^{f,L}(\mathcal{T}_{g,n}/\text{Stab}(\gamma)) = \frac{1}{[\text{Stab}(\gamma):\text{Stab}_0(\gamma)]} \cdot \dot{\mu}_\gamma^{f,L}(\mathcal{T}_{g,n}/\text{Stab}_0(\gamma)).
	\]
	Recall that by definition
	\[
	d \mu_\gamma^{f,L}(X) := f\left(\textstyle \frac{1}{L} \cdot (\ell_{\gamma_i}(X))_{i=1}^k\right) \ d\mu_{\text{wp}}(X).
	\]  
	In particular, identifying points in $\mathcal{T}_{g,n}/\text{Stab}_0(\gamma)$ with pairs $(X,\alpha)$ as in the discussion preceding Theorem \ref{theo:cut_glue_fib},
	\[
	d \dot{\mu}_\gamma^{f,L}(X,\alpha) = f(\textstyle \frac{1}{L} \cdot (\ell_{\alpha_i}(X))_{i=1}^k) \ d\dot{\mu}_{\text{wp}}(X),
	\]
	where $\dot{\mu}_{\text{wp}}$ is the local pushforward to $\mathcal{T}_{g,n}/\text{Stab}_0(\gamma)$ of the Weil-Petersson measure $\mu_{\text{wp}}$ on $\mathcal{T}_{g,n}$ . It follows from Theorem \ref{theo:cut_glue_fib} that
	\[
	\widetilde{\mu}_\gamma^{f,L}(\mathcal{T}_{g,n}/\text{Stab}(\gamma)) = \int_{\mathbf{R}^k} f(\textstyle\frac{1}{L} \cdot \mathbf{L}) \cdot V_{g,n}(\gamma,\mathbf{L}) \ d\mathbf{L}.
	\]
	The change of variable $\mathbf{u} := \frac{1}{L} \cdot \mathbf{L}$ yields
	\[
	m_\gamma^{f,L} = \int_{\mathbf{R}^k} f(\mathbf{u}) \cdot V_{g,n}(\gamma,L \cdot \mathbf{u}) \cdot L^k \  d \mathbf{u}.
	\]
	As the function $f$ is compactly supported and bounded, and as $V_{g,n}(\gamma,\mathbf{u})$ is a polynomial of degree $6g-6+2n-k$ on the $\mathbf{u}$ variables, the dominate convergence theorem ensures
	\begin{align*}
	\lim_{L \to \infty} \frac{m_\gamma^{f,L}}{L^{6g-6+2n}} = \int_{\mathbf{R}^k} f(\mathbf{u}) \cdot W_{g,n}(\gamma,\mathbf{u}) \ d \mathbf{u}.
	\end{align*}
	This finishes the proof.
\end{proof}
$ $

\textit{Earthquake flow invariance.} Proposition \ref{prop:earth_inv_0} is a  direct consequence of the continuity of the earthquake flow on $P^1\mathcal{M}_{g,n}$ and the following result.\\

\begin{proposition}
	\label{prop:inv_meas}
	The measures $\widehat{\nu}_{\gamma,\mathbf{a}}^{f,L}$ on $P^1 \mathcal{M}_g$ are earthquake flow invariant.\\
\end{proposition}

\begin{proof}
	Fix $L > 0$. We first check that the measure $\nu_{\gamma,\mathbf{a}}^{f,L}$ on $P^1 \mathcal{T}_{g,n}$ is earthquake flow invariant. To simplify the notation,  consider the function $F \colon \mathcal{T}_{g,n} \to \mathbf{R}_{\geq 0}$ which to every $X \in \mathcal{T}_{g,n}$ assigns the value
	\[
	F(X) := f\left(\textstyle \frac{1}{L} \cdot (\ell_{\gamma_i}(X))_{i=1}^k\right),
	\]
	so that 
	\[
	d \mu_\gamma^{f,L}(X) =  F(X) \ d\mu_{\text{wp}}(X).
	\]  
	Notice that $F$ is invariant under twists along $\mathbf{a} \cdot \gamma \in \mathcal{ML}_{g,n}(\mathbf{Q})$ because the length functions $\ell_{\gamma_i} \colon \mathcal{T}_{g,n} \allowbreak \to \mathbf{R}_{>0}$ are invariant under such twists; this is the only fact about the function $F$ we will need. \\
	
	Recall the disintegration formula
	\[
	d \nu_{\gamma,\mathbf{a}}^{f,L}(X,\lambda) := d \delta_{\mathbf{a} \cdot \gamma/ \ell_{\mathbf{a} \cdot\gamma}(X)}(\lambda) \ d\mu_\gamma^{f,L}(X) = d \delta_{\mathbf{a} \cdot \gamma/ \ell_{\mathbf{a} \cdot \gamma}(X)}(\lambda) \  F(X) \ d\mu_{\text{wp}}(X).
	\]
	By Proposition \ref{prop:meas_disint},
	\[
	d \mu_{\text{wp}}(X) = d \eta_{\mathbf{a}\cdot\gamma}^r(X) \ dr,
	\]
	where $\eta_{\mathbf{a}\cdot\gamma}^r$ is the horospherical measure on $\mathcal{T}_{g,n}$ corresponding to $\mathbf{a} \cdot \gamma \in \mathcal{ML}_{g,n}(\mathbf{Q})$ and $r > 0$. It follows that
	\[
	d \delta_{\mathbf{a} \cdot \gamma/ \ell_{\mathbf{a} \cdot \gamma}(X)}(\lambda) \  F(X) \ d\mu_{\text{wp}}(X) = d \delta_{\mathbf{a} \cdot \gamma/ \ell_{\mathbf{a} \cdot \gamma}(X)}(\lambda) \  F(X) \ d\eta_{\mathbf{a}\cdot\gamma}^r(X) \ dr. 
	\]
	As $\eta_{\mathbf{a}\cdot\gamma}^r$ is supported on the horosphere $S_{\mathbf{a} \cdot \gamma}^r := \ell_{\mathbf{a} \cdot \gamma}^{-1}(r)$,
	\[
	d \delta_{\mathbf{a} \cdot \gamma/ \ell_{\mathbf{a} \cdot \gamma}(X)}(\lambda) \  F(X) \ d\eta_{\mathbf{a}\cdot\gamma}^r(X) \ dr  = d \delta_{\mathbf{a} \cdot \gamma/r}(\lambda) \  F(X) \ d\eta_{\mathbf{a}\cdot\gamma}^r(X) \ dr.
	\]
	By Fubini's theorem,
	\[
	d \delta_{\mathbf{a} \cdot \gamma/r}(\lambda) \  F(X) \ d\eta_{\mathbf{a}\cdot\gamma}^r(X) \ dr = F(X) \ d\eta_{\mathbf{a}\cdot\gamma}^r(X) \ d \delta_{\mathbf{a} \cdot \gamma/r}(\lambda) \  dr.
	\]
	Putting things together we deduce
	\begin{equation}
	\label{eq:conv_disint}
	d \nu_{\gamma,\mathbf{a}}^{f,L}(X,\lambda) = F(X) \ d\eta_{\mathbf{a}\cdot\gamma}^r(X) \ d \delta_{\mathbf{a} \cdot \gamma/r}(\lambda) \  dr.
	\end{equation}
	$ $
	
	Let $t \in \mathbf{R}$ be arbitrary. From (\ref{eq:conv_disint}) and the fact that the function $F$ and the measures $\eta_{\mathbf{a}\cdot\gamma}^r$ are invariant under twists along $\mathbf{a} \cdot \gamma \in \mathcal{ML}_{g,n}(\mathbf{Q})$ we deduce 
	\begin{align*}
	\text{tw}^t_*(d \nu_{\gamma,\mathbf{a}}^{f,L})(X,\lambda) &= d \nu_{\gamma,\mathbf{a}}^{f,L}(\text{tw}^{-t}(X,\lambda)) \\
	&= d \nu_{\gamma,\mathbf{a}}^{f,L}(\text{tw}_\lambda^{-t}(X),\lambda) \\
	&=F(\text{tw}_\lambda^{-t}(X)) \ d\eta_{\mathbf{a}\cdot\gamma}^r(\text{tw}_\lambda^{-t}(X)) \ d \delta_{\mathbf{a} \cdot \gamma/r}(\lambda) \  dr \\
	&=F(\text{tw}_{\mathbf{a} \cdot \gamma/r}^{-t}(X)) \ d\eta_{\mathbf{a}\cdot\gamma}^r(\text{tw}_{\mathbf{a} \cdot \gamma/r}^{-t}(X)) \ d \delta_{\mathbf{a} \cdot \gamma/r}(\lambda) \  dr\\
	&=F(X) \ d\eta_{\mathbf{a}\cdot\gamma}^r(X) \ d \delta_{\mathbf{a} \cdot \gamma/r}(\lambda) \  dr \\
	&= d \nu_{\gamma,\mathbf{a}}^{f,L}(X,\lambda).
	\end{align*}
	We have thus shown that the measure $\nu_{\gamma,\mathbf{a}}^{f,L}$ on $P^1 \mathcal{T}_{g,n}$ is earthquake flow invariant. \\
	
	Recall that the action of $\text{Mod}_{g,n}$ on $P^1 \mathcal{T}_{g,n}$ commutes with the earthquake flow. As $\widetilde{\nu}_{\gamma,\mathbf{a}}^{f,L}$ is the local pushforward to $\mathcal{T}_{g,n}/\text{Stab}(\gamma)$ of the measure $\nu_{\gamma,\mathbf{a}}^{f,L}$ on $\mathcal{T}_{g,n}$, it follows that $\widetilde{\nu}_{\gamma,\mathbf{a}}^{f,L}$ is earthquake flow invariant. As the earthquake flow commutes with the quotient map $P^1 \mathcal{T}_{g,n}/\text{Stab}(\gamma) \to P^1\mathcal{M}_{g,n}$ and as $\widehat{\nu}_{\gamma,\mathbf{a}}^{f,L}$ is the pushforward of the measure $\widetilde{\nu}_{\gamma,\mathbf{a}}^{f,L}$ under this map, it follows that $\widehat{\nu}_{\gamma,\mathbf{a}}^{f,L}$ is earthquake flow invariant. This finishes the proof. 
\end{proof}
$ $

\textit{Comparing horoball segment measures.} Many of the proofs that follow can be reduced to the study of the horoball segment measures associated to the indicator functions 
\[
\mathbbm{1}_{B_a} \colon \left(\mathbf{R}_{\geq 0}\right)^k \to \mathbf{R}_{\geq 0}
\]
of boxes $B_a \subseteq (\mathbf{R}_{\geq 0})^k$ of the form
\[
B_a := [0,a]^k,
\]
where $a > 0$ is arbitrary. For every $L > 0$, let $B_\gamma^{a,L} \subseteq \mathcal{T}_{g,n}$ be the corresponding horoball segment, $\mu_\gamma^{a,L}$ be the corresponding horoball segment measure on $\mathcal{T}_{g,n}$, $\widetilde{\mu}_\gamma^{a,L}$ be the local pushforward of $\mu_\gamma^{a,L}$ to $\mathcal{T}_{g,n}/\text{Stab}(\gamma)$, and $\widehat{\mu}_\gamma^{a,L}$ be the pushforward to $\mathcal{M}_{g,n}$ of $\widetilde{\mu}_\gamma^{a,L}$. Analogously, for every $L > 0$, let  $\nu_{\gamma,\mathbf{a}}^{a,L}$ be the corresponding horoball segment measure on $P^1\mathcal{T}_{g,n}$, $\widetilde{\nu}_{\gamma,\mathbf{a}}^{a,L}$ be the local pushforward of $\nu_{\gamma,\mathbf{a}}^{a,L}$ to $P^1\mathcal{T}_{g,n}/\text{Stab}(\gamma)$, and $\widehat{\nu}_{\gamma,\mathbf{a}}^{a,L}$ be the pushforward to $P^1\mathcal{M}_{g,n}$ of $\widetilde{\nu}_{\gamma,\mathbf{a}}^{a,L}$. Let $m_\gamma^{a,L}$ be the total mass of the measures $\widehat{\mu}_\gamma^{a,L}$ and $\widehat{\nu}_{\gamma,\mathbf{a}}^{a,L}$.\\

The following lemma will allow us to carry out the reductions mentioned above.\\

\begin{lemma}
	\label{lemma:hor_meas_comp}
	There exist constants $a>0$ and $C > 0$ such that for every Borel measurable subset $A \subseteq P^1 \mathcal{M}_{g,n}$,
	\[
	\limsup_{L \to \infty} \frac{\widehat{\nu}_{\gamma,\mathbf{a}}^{f,L}(A)}{m_{\gamma}^{f,L}} \leq C \cdot 	\limsup_{L \to \infty} \frac{\widehat{\nu}_{\gamma,\mathbf{a}}^{a,L}(A)}{m_{\gamma}^{a,L}}.
	\]
\end{lemma}
$ $

\begin{proof}
	Let $a > 0$ be such that the support of $f$ is contained in the box $B_a \subseteq (\mathbf{R}_{\geq0})^k$ and $M > 0$ be such that $f \leq M$. It follows directly from the definitions that
	\[
	\nu_{\gamma,\mathbf{a}}^{f,L} \leq M \cdot \nu_{\gamma,\mathbf{a}}^{a,L}.
	\]
	In particular, after considering the relevant local pushforwards and pushforwards,
	\[
	\widehat{\nu}_{\gamma,\mathbf{a}}^{f,L} \leq M \cdot \widehat{\nu}_{\gamma,\mathbf{a}}^{a,L}.
	\]
	$ $
	As a consequence of Proposition \ref{prop:total_hor_meas},
	\[
	\lim_{L \to \infty} \frac{m_\gamma^{f,L}}{m_\gamma^{a,L}} = c
	\]
	for some positive constant $c > 0$. It follows that, for every Borel measurable subset $A \subseteq P^1\mathcal{M}_{g,n}$,
	\[
	\limsup_{L \to \infty} \frac{\widehat{\nu}_{\gamma,\mathbf{a}}^{f,L}(A)}{m_{\gamma}^{f,L}} \leq M \cdot 	\limsup_{L \to \infty} \frac{\widehat{\nu}_{\gamma,\mathbf{a}}^{a,L}(A)}{m_{\gamma}^{f,L}} \leq M \cdot c \cdot 	\limsup_{L \to \infty} \frac{\widehat{\nu}_{\gamma,\mathbf{a}}^{a,L}(A)}{m_{\gamma}^{a,L}}.
	\]
	Letting $C:= M \cdot c$ finishes the proof.
\end{proof}
$ $

\textit{Absolute continuity with respect to the Mirzakhani measure.} Recall that $U_X(\epsilon) \subseteq \mathcal{T}_{g,n}$ denotes the open ball of radius $\epsilon > 0$ centered at $X \in \mathcal{T}_{g,n}$ with respect to the symmetric Thurston metric $d_\text{Thu}$. A subset of a topological space is said to be a \textit{continuity subset} of a given measure class if its boundary has measure zero with respect to such measure class. To prove Proposition \ref{prop:hor_meas_ac_0} we make use the following technical result, which is a consequence of the outer regularity of the measure $\nu_\text{Mir}$ on $P^1 \mathcal{T}_{g,n}$, the Vitali covering lemma, Lemma \ref{lemma:thu_ball_bound}, and the fact that the Thurston measure $\mu_{\text{Thu}}$ on $\mathcal{ML}_{g,n}$ is, up to a constant, the Lebesgue measure on train track coordinates.\\

\begin{lemma}
	\label{lemma:approx}
	Let $K \subseteq P^1\mathcal{T}_{g,n}$ be a compact subset. There exists a constant $\epsilon_0 > 0$ such that for every Borel measurable subset $A \subseteq K$ with $\nu_{\text{Mir}}(A) = 0$ and every $\delta> 0$, there exists a countable cover
	\[
	A \subseteq \bigcup_{i\in \mathbf{N}} U_{X_i}(\epsilon_i) \times V_i
	\]
	with $X_i \in K$, $0 < \epsilon_i < \epsilon_0$, $V_i \subseteq P\mathcal{ML}_{g,n}$ open continuity subset of the Lebesgue measure class, and such that
	\[
	\sum_{i\in \mathbf{N}} \nu_{\text{Mir}}(U_{X_i}(\epsilon_i) \times V_i)< \delta.
	\]
\end{lemma}
$ $

Let $\Pi \colon P^1 \mathcal{T}_{g,n} \to P^1 \mathcal{M}_{g,n}$ be the quotient map induced by the mapping class group action on $P^1\mathcal{T}_{g,n}$. Using Lemma \ref{lemma:approx} we will reduce the proof of Proposition \ref{prop:hor_meas_ac_0} to the following estimate.\\

\begin{proposition}
	\label{prop:ac_key_estimate}
	Let $K \subseteq \mathcal{T}_{g,n}$ be a compact subset and $\epsilon_0 > 0$ be fixed. There exists a constant $C > 0$ such that for every $X \in K$, every $0 < \epsilon < \epsilon_0$, and every $V \subseteq P \mathcal{ML}_{g,n}$ open continuity subset of the Lebesgue measure class,
	\[
	\limsup_{L \to \infty} \frac{\widehat{\nu}_{\gamma,\mathbf{a}}^{f,L}(\Pi(U_X(\epsilon) \times V))}{m_\gamma^{f,L}} \leq C \cdot \nu_{\text{Mir}}(U_X(\epsilon) \times V).
	\]
\end{proposition}
$ $

Let us prove Proposition \ref{prop:hor_meas_ac_0} using Lemma \ref{lemma:approx} and Proposition \ref{prop:ac_key_estimate}.\\

\begin{proof}[Proof of Proposition \ref{prop:hor_meas_ac_0}]
	Let $L_j \nearrow +\infty$ be an increasing sequence of positive real numbers such that for some finite measure $\widehat{\nu}_{\gamma,\mathbf{a}}^f$ on $P^1\mathcal{M}_{g,n}$,
	\[
	\lim_{j \to \infty} \frac{\widehat{\nu}_{\gamma,\mathbf{a}}^{f,L_j}}{m_{\gamma}^{f,L_j}} = \widehat{\nu}_{\gamma,\mathbf{a}}^f
	\] 
	in the weak-$\star$ topology. Let $\widehat{A} \subseteq P^1 \mathcal{M}_{g,n}$ be a Borel measurable subset such that $\widehat{\nu}_{\text{Mir}}(\widehat{A})=0$. Our goal is to show that $\widehat{\nu}_{\gamma,\mathbf{a}}^f(\widehat{A}) = 0$. As $P^1 \mathcal{M}_{g,n}$ admits a countable exhaustion by compact sets and as the limit measure $\widehat{\nu}_{\gamma,\mathbf{a}}^f$ is continuous with respect to increasing limits of sets, we can assume without loss of generality that $\widehat{A} \subseteq \widehat{K}$ for some compact subset $\widehat{K} \subseteq P^1 \mathcal{M}_{g,n}$. Let $K \subseteq P^1 \mathcal{T}_{g,n}$ be a compact subset covering $\widehat{K}$ and $A \subseteq K$ be a subset covering $\widehat{A}$. Notice that $\nu_{\text{Mir}}(A) = 0$.\\
	
	Let $\delta > 0$ be arbitrary. Lemma \ref{lemma:approx} provides a countable cover
	\[
	A \subseteq \bigcup_{i \in \mathbf{N}} U_{X_i}(\epsilon_i) \times V_i,
	\]
	with $X_i \in K$, $0 < \epsilon_i < \epsilon_0$, $V_i \subseteq P\mathcal{ML}_{g,n}$ open continuity subset of the Lebesgue measure class, and such that
	\[
	\sum_{i\in \mathbf{N}} \nu_{\text{Mir}}(U_{X_i}(\epsilon_i) \times V_i)< \delta.
	\]
	Monotonicity of the limit measure $\widehat{\nu}_{\gamma,\mathbf{a}}^f$ ensures
	\[
	\widehat{\nu}_{\gamma,\mathbf{a}}^f(\widehat{A}) \leq \widehat{\nu}_{\gamma,\mathbf{a}}^f \left(\bigcup_{i\in \mathbf{N}} \Pi(U_{X_i}(\epsilon_i) \times V_i)\right).
	\]
	As the limit measure $\widehat{\nu}_{\gamma,\mathbf{a}}^f$ is finite and continuous with respect to increasing limits of sets, we can find a finite subset $I \subseteq\mathbf{N}$ such that
	\[
	\widehat{\nu}_{\gamma,\mathbf{a}}^f \left(\bigcup_{i\in \mathbf{N}} \Pi(U_{X_i}(\epsilon_i) \times V_i)\right) \leq \widehat{\nu}_{\gamma,\mathbf{a}}^f \left(\bigcup_{i\in I} \Pi(U_{X_i}(\epsilon_i) \times V_i)\right) + \delta.
	\]
	Portmanteau's theorem applied to open sets with compact closure ensures 
	\[
	\widehat{\nu}_{\gamma,\mathbf{a}}^f \left(\bigcup_{i \in I} \Pi(U_{X_i}(\epsilon_i) \times V_i)\right)  \leq \liminf_{j \to \infty} \ \frac{\widehat{\nu}_{\gamma,\mathbf{a}}^{f,L_j}}{m_\gamma^{f,L_j}} \left(\bigcup_{i\in I} \Pi(U_{X_i}(\epsilon_i) \times V_i)\right).
	\]
	By subadditivity of the measures $\widehat{\nu}_{\gamma,\mathbf{a}}^{f,L_j}/m_\gamma^{f,L_j}$,
	\[
	\frac{\widehat{\nu}_{\gamma,\mathbf{a}}^{f,L_j}}{m_\gamma^{f,L_j}} \left(\bigcup_{i\in I} \Pi(U_{X_i}(\epsilon_i) \times V_i)\right)  \leq \sum_{i\in I} \frac{\widehat{\nu}_{\gamma,\mathbf{a}}^{f,L_j}}{m_\gamma^{f,L_j}} \left( \Pi(U_{X_i}(\epsilon_i) \times V_i)\right).
	\]
	Proposition \ref{prop:ac_key_estimate} provides a constant $C > 0$ such that for every $i \in I$,
	\[
	\limsup_{j \to \infty} \frac{\widehat{\nu}_{\gamma,\mathbf{a}}^{f,L_j}}{m_\gamma^{f,L_j}} \left( \Pi(U_{X_i}(\epsilon_i) \times V_i )\right) \leq C \cdot \nu_{\text{Mir}}(U_{X_i}(\epsilon_i) \times V_i).
	\]
	In particular, 
	\[
	\limsup_{j \to \infty} \sum_{i\in I} \frac{\widehat{\nu}_{\gamma,\mathbf{a}}^{f,L_j}}{m_\gamma^{f,L_j}} \left( \Pi(U_{X_i}(\epsilon_i) \times V_i)\right)  \leq C \cdot \sum_{i \in I} \nu_{\text{Mir}}(U_{X_i}(\epsilon_i) \times V_i).
	\]
	Putting things together we deduce
	\[
	\widehat{\nu}_{\gamma,\mathbf{a}}^f(\widehat{A}) \leq (1 + C) \cdot \delta.
	\]
	As $\delta  > 0$ is arbitrary,  $\widehat{\nu}_{\gamma,\mathbf{a}}^f(\widehat{A})  = 0$. This finishes the proof.
\end{proof}
$ $

We now prove Proposition \ref{prop:ac_key_estimate}. By Lemma \ref{lemma:hor_meas_comp}, it is enough to show this proposition holds for the sequence of probability measures $\{\widehat{\nu}_{\gamma,\mathbf{a}}^{a,L}/m_\gamma^{a,L}\}_{L > 0}$ with $a > 0$ arbitrary. For the rest of this discussion fix $a>0$.\\

We prove Proposition \ref{prop:ac_key_estimate} in three steps. Given $\alpha \in \mathcal{ML}_{g,n}$, denote by $[\alpha]$ its equivalence class in $P\mathcal{ML}_{g,n}$. Let $X \in \mathcal{T}_{g,n}$ and $V \subseteq P \mathcal{ML}_{g,n}$ be a Borel measurable subset. For every $L > 0$ consider the counting function
\[
s(X,\mathbf{a} \cdot \gamma,L,V) := \# \{\alpha \in \text{Mod}_{g,n} \cdot (\mathbf{a} \cdot \gamma) \ | \ \ell_{\alpha}(X) \leq L, \ [\alpha] \in V \},
\]
where $\mathbf{a} \cdot \gamma \in \mathcal{ML}_{g,n}(\mathbf{Q})$ is as in (\ref{eq:weight_curve}). Notice that
\[
s(X,\mathbf{a} \cdot \gamma,L,V) = \# \{\phi \in \text{Mod}_{g,n}/\text{Stab}(\mathbf{a} \cdot \gamma) \ | \ \ell_{\phi \cdot (\mathbf{a} \cdot \gamma)}(X) \leq L, \ [\phi \cdot (\mathbf{a} \cdot \gamma)] \in V \}.
\]
Denote
\[
S(\mathbf{a}) := a_1+\cdots+a_k.
\]
The following lemma is the first step in the proof of Proposition \ref{prop:ac_key_estimate}.\\

\begin{lemma}
	\label{lemma:quot_bound}
	For every $X \in \mathcal{T}_{g,n}$, every $\epsilon > 0$, every Borel measurable subset $V \subseteq P\mathcal{ML}_{g,n}$, and every $L > 0$, 
	\[
	\widehat{\nu}_{\gamma,\mathbf{a}}^{a,L}(\Pi(U_X(\epsilon) \times V)) \leq [\text{Stab}(\mathbf{a} \cdot \gamma):\text{Stab}(\gamma)]\cdot s(X,\gamma,S(\mathbf{a})e^\epsilon aL,V) \cdot \mu_{\text{wp}}(U_X(\epsilon)).
	\]
\end{lemma}
$ $

\begin{proof}
	Consider the quotient map $P \colon P^1 \mathcal{T}_{g,n} / \text{Stab}(\mathcal{\gamma}) \to P^1 \mathcal{M}_{g,n}$. Recall that $\widehat{\nu}_{\gamma,\mathbf{a}}^{a,L}$ is the pushforward of the measure $\widetilde{\nu}_{\gamma,\mathbf{a}}^{a,L}$ on $P^1 \mathcal{T}_{g,n}/ \text{Stab}(\gamma)$ under the map $P$. It follows that
	\[
	\widehat{\nu}_{\gamma,\mathbf{a}}^{a,L}(\Pi(U_X(\epsilon) \times V)) = \widetilde{\nu}_{\gamma,\mathbf{a}}^{a,L}(P^{-1}(\Pi(U_X(\epsilon) \times V))).
	\]
	Consider the quotient map $Q \colon P^1 \mathcal{T}_{g,n} \to P^1 \mathcal{T}_{g,n}/\text{Stab}(\gamma)$. Notice that $\Pi = P \circ Q$. It follows that
	\[
	P^{-1}(\Pi(U_X(\epsilon) \times V)) = Q(\text{Mod}_{g,n} \cdot (U_X(\epsilon) \times V)).
	\]
	To describe $Q(\text{Mod}_{g,n} \cdot (U_X(\epsilon) \times V))$ it is enough to consider one mapping class in every left coset of $\text{Stab}(\mathcal{\gamma}) \backslash \text{Mod}_{g,n}$ rather than all mapping classes in $\text{Mod}_{g,n}$. It follows that
	\[
	P^{-1}(\Pi(U_X(\epsilon) \times V)) = \bigcup_{[\phi] \in \text{Stab}(\mathcal{\gamma}) \backslash \text{Mod}_{g,n}} Q(\phi \cdot (U_X(\epsilon) \times V)).
	\]
	Using the subadditivity of the measure $\widetilde{\nu}_{\gamma,\mathbf{a}}^{a,L}$ we deduce
	\[
	\widetilde{\nu}_{\gamma,\mathbf{a}}^{a,L}(P^{-1}(\Pi(U_X(\epsilon) \times V))) \leq \sum_{[\phi] \in \text{Stab}(\gamma) \backslash \text{Mod}_{g,n}} \widetilde{\nu}_{\gamma,\mathbf{a}}^{a,L}(Q(\phi \cdot (U_X(\epsilon) \times V))).
	\]
	Using the bijection $\text{Stab}(\gamma) \backslash \text{Mod}_{g,n} \to \text{Mod}_{g,n}/\text{Stab}(\gamma)$ given by $[\phi] \mapsto [\phi^{-1}]$, we can rewrite the above inequality as
	\begin{equation}
	\label{eq:measure_sum}
	\widetilde{\nu}_{\gamma,\mathbf{a}}^{a,L}(P^{-1}(\Pi(U_X(\epsilon) \times V))) \leq \sum_{[\phi] \in \text{Mod}_{g,n}/\text{Stab}(\gamma) } \widetilde{\nu}_{\gamma,\mathbf{a}}^{a,L}(Q(\phi^{-1} \cdot (U_X(\epsilon) \times V))).
	\end{equation}
	Recall that $\widetilde{\nu}_{\gamma,\mathbf{a}}^{a,L}$ is the local pushforward of the measure $\nu_{\gamma,\mathbf{a}}^{a,L}$ on $P^1\mathcal{T}_{g,n}$ under the quotient map $Q \colon P^1 \mathcal{T}_{g,n} \to P^1 \mathcal{T}_{g,n}/\text{Stab}(\gamma)$. It follows that
	\[
	\widetilde{\nu}_{\gamma,\mathbf{a}}^{a,L}(Q(\phi^{-1} \cdot (U_X(\epsilon) \times V))) \leq \nu_{\gamma,\mathbf{a}}^{a,L}(\phi^{-1} \cdot (U_X(\epsilon) \times V))
	\]
	for every $[\phi] \in \text{Mod}_g/\text{Stab}(\gamma)$. Directly from the definition of the measures $\nu_{\gamma,\mathbf{a}}^{a,L}$ and the fact that the Weil-Petersson measure $\mu_{\text{wp}}$ on $\mathcal{T}_{g,n}$ is $\text{Mod}_{g,n}$-invariant, one can show that
	\[
	\nu_{\gamma,\mathbf{a}}^{a,L}(A) = \nu_{\phi \cdot \gamma,\mathbf{a}}^{a,L}(\phi \cdot A)
	\]
	for every Borel measurable set $A \subseteq P^1 \mathcal{T}_{g,n}$ and every $\phi \in \text{Mod}_{g,n}$. In particular,
	\[
	\nu_{\gamma,\mathbf{a}}^{a,L}(\phi^{-1} \cdot (U_X(\epsilon) \times V)) = \nu_{\phi \cdot \gamma, \mathbf{a}}^{a,L}( U_X(\epsilon) \times V)
	\]
	for every $[\phi] \in \text{Mod}_g/\text{Stab}(\gamma)$. Notice that if $\nu_{\phi \cdot \gamma,\mathbf{a}}^{a,L}(U_X(\epsilon) \times V) \neq 0$ then $[\phi \cdot (\mathbf{a} \cdot \gamma)] \in V$ and $B_{\phi \cdot \gamma}^{a,L} \cap U_X(\epsilon) \neq \emptyset$. By the definition of the horoball segment $B_{\phi \cdot \gamma}^{a,L}$ and of the symmetric Thurston metric $d_{\text{Thu}}$, the second condition implies
	\[
	\ell_{\phi \cdot \gamma_i}(X) \leq e^\epsilon a L, \ \forall i = 1 , \dots, k.
	\]
	In particular, $\ell_{\phi \cdot (\mathbf{a} \cdot \gamma)}(X) \leq S(\mathbf{a})e^\epsilon aL$. It follows that in the sum on the right hand side of (\ref{eq:measure_sum}) at most $[\text{Stab}(\mathbf{a} \cdot \gamma):\text{Stab}(\gamma)] \cdot s(X,\gamma,S(\mathbf{a})e^\epsilon aL,V)$ summands are non-zero. For each one of these summands, 
	\[
	\nu_{\phi \cdot \gamma,\mathbf{a}}^{a,L}(U_X(\epsilon) \times V) \leq \mu_{\text{wp}}(U_X(\epsilon))
	\]
	because $d\nu_{\phi \cdot \gamma,\mathbf{a}}^{a,L} (X,\lambda) = d \delta_{\mathbf{a} \cdot (\phi \cdot \gamma)/\ell_{\mathbf{a} \cdot (\phi \cdot \gamma)}(X)}(\lambda) \ d\mu_{\phi \cdot \gamma}^{a,L}(X)$ and $\mu_{\phi \cdot \gamma}^{a,L}$ is the restriction of $\mu_\text{wp}$ to the horoball $H_{\phi \cdot \gamma}^{a,L}$. Putting everything together we deduce
	\[
	\widehat{\nu}_{\gamma,\mathbf{a}}^{a,L}(\Pi(U_X(\epsilon) \times V)) \leq  [\text{Stab}(\mathbf{a} \cdot \gamma):\text{Stab}(\gamma)] \cdot s(X,\mathbf{a} \cdot \gamma,S(\mathbf{a})e^\epsilon aL,V) \cdot \mu_{wp}(U_X(\epsilon)).
	\]
	This finishes the proof.
\end{proof}
$ $

Given a Borel measurable subset $V \subseteq P\mathcal{ML}_{g,n}$, let $B_V \colon \mathcal{T}_{g,n} \to \mathbf{R}_{\geq 0}$ be the function which to every $X \in \mathcal{T}_{g,n}$ assigns the value
\begin{equation}
\label{eq:bv_def_0}
B_V(X) := \mu_{\text{Thu}}\left( \{\lambda \in \mathcal{ML}_{g,n} \colon \ell_\lambda(X) \leq L, \ [\lambda] \in V\}\right).
\end{equation}
The following lemma is the second step in the proof of Proposition \ref{prop:ac_key_estimate}. It is a consequence of Portmanteau's theorem applied to the definition of the Thurston measure $\mu_{\text{Thu}}$ on $\mathcal{ML}_{g,n}$ in (\ref{ML_counting_measure}); it is important for this step that the positive weights $\mathbf{a} \in (\mathbf{Q}_{>0})^k$ are rational.\\

\begin{lemma}
	\label{lemma:count_bd_ac}
	There exists a constant $C >0$ such that for every $X \in \mathcal{T}_{g,n}$ and every $V \subseteq P \mathcal{ML}_{g,n}$ open continuity subset of the Lebesgue measure class, 
	\[
	\limsup_{L \to \infty} \frac{s(X,\mathbf{a} \cdot \gamma,L,V)}{L^{6g-6+2n}} \leq C \cdot B_V(X).
	\]
\end{lemma}
$ $

The last step in the proof of Proposition \ref{prop:ac_key_estimate} corresponds to the following lemma, which follows from standard compactness arguments and the fact that the Thurston measure $\mu_{\text{Thu}}$ on $\mathcal{ML}_{g,n}$ is, up to a constant, the Lebesgue measure on train track coordinates.\\

\begin{lemma}
	\label{lemma:bv_measure_comparison}
	Let $K \subseteq \mathcal{T}_{g,n}$ be a compact subset and $\epsilon_0 > 0$ be fixed. There exists a constant $C > 0$ such that for every $X \in K$, every $0 < \epsilon < \epsilon_0$, and every Borel measurable subset $V \subseteq P \mathcal{ML}_{g,n}$, 
	\[
	\mu_{\text{wp}}(U_X(\epsilon)) \cdot B_V(X) \leq C \cdot \nu_{\text{Mir}}(U_X(\epsilon) \times V).
	\]
\end{lemma}
$ $

Proposition \ref{prop:ac_key_estimate} now follows directly from Lemmas \ref{lemma:quot_bound}, \ref{lemma:count_bd_ac}, and \ref{lemma:bv_measure_comparison}, and Proposition \ref{prop:total_hor_meas}.\\

\textit{No escape of mass.} Recall that, for every $\epsilon > 0$, $K_\epsilon \subseteq \mathcal{M}_{g,n}$ denotes the $\epsilon$-thick part of $\mathcal{M}_{g,n}$ and $P^1 K_\epsilon \subseteq P^1\mathcal{M}_{g,n}$ denotes the $\epsilon$-thick part of $P^1 \mathcal{M}_{g,n}$; as mentioned in \S 1, these sets are compact. To prove Proposition \ref{prop:hor_meas_nem_0}
we show that the sequence of probability measures $\{\widehat{\nu}_{\gamma,\mathbf{a}}^{f,L}/m_\gamma^{f,L}\}_{L > 0}$ on $P^1\mathcal{M}_{g,n}$ exhibits no escape of mass; this could potentially happen given that $P^1\mathcal{M}_{g,n}$ is not compact. More specifically, Proposition \ref{prop:hor_meas_nem_0} can be deduced directly from the following result.\\

\begin{proposition}
	\label{prop:hor_meas_qnem}
	For every $\delta > 0$ there exists $\epsilon > 0$ such that
	\[
	\liminf_{L \to \infty} \frac{\widehat{\nu}_{\gamma,\mathbf{a}}^{f,L}(P^1 K_\epsilon)}{m_\gamma^{f,L}} \geq 1 - \delta.
	\]
\end{proposition}
$ $

We devote the rest of this section to proving Proposition \ref{prop:hor_meas_qnem}. By Lemma \ref{lemma:hor_meas_comp}, it is enough to show this proposition holds 
for the sequence of probability measures $\{\widehat{\nu}_{\gamma,\mathbf{a}}^{a,L}/{m_\gamma^{a,L}}\}_{L>0}$ with $a > 0$ arbitrary. For the rest of this discussion fix $a > 0$. \\

Directly form the definitions we see that
\[
\widehat{\nu}_{\gamma,\mathbf{a}}^{a,L}(P^1 K_\epsilon) = \widehat{\mu}_\gamma^{a,L}(K_\epsilon)
\]
for every $\epsilon > 0$ and every $L > 0$. It is enough then to prove that for every $\delta > 0$ there exists $\epsilon > 0$ such that
\[
\liminf_{L \to \infty} \frac{\widehat{\mu}_\gamma^{a,L}( K_\epsilon)}{m_\gamma^{a,L}} \geq 1 - \delta,
\]
i.e., we can reduce ourselves to working on the base $\mathcal{M}_{g,n}$ of the bundle $P^1\mathcal{M}_{g,n}$.\\

To estimate $\widehat{\mu}_\gamma^{a,L}(K_\epsilon)$ we use the cut and glue fibration 
\[
\Psi \colon \mathcal{T}_{g,n}/\text{Stab}_0(\gamma) \to \Omega_{g,n}(\gamma);
\]
see Theorem \ref{theo:cut_glue_fib}. To control the measure along the fibers of $\Psi$ we use Theorem \ref{theo:mw_quant_rec}. We refer the reader to Theorem \ref{theo:cut_glue_fib} and the discussion preceding it for the notation and terminology that will be used in the arguments that follow. \\

Given a triple of non-negative integers $g',n',b' \in \mathbf{Z}_{\geq 0}$ satisfying $2 - 2g' - n' - b' < 0$ and a vector $\mathbf{L}:= (L_i)_{i=1}^{b'} \in (\mathbf{R}_{>0})^{b'}$ of positive real numbers, denote by 
\[
\mathcal{M}_{g',n'}^{b'}(\mathbf{L})_\epsilon \subseteq \mathcal{M}_{g',n'}^{b'}(\mathbf{L})
\]
the subset of all hyperbolic surfaces in $\mathcal{M}_{g',n'}^{b'}(\mathbf{L})$ that have no (non boundary parallel) geodesics of length $< \epsilon$. We refer to $\mathcal{M}_{g',n'}^{b'}(\mathbf{L})_\epsilon$ as the \textit{$\epsilon$-thick part} of $\mathcal{M}_{g',n'}^{b'}(\mathbf{L})$.  \\

Given $(X,\alpha) \in \mathcal{T}_{g,n}/\text{Stab}_0(\gamma)$, to every component $\alpha_i$ of $\alpha$ we associate a \textit{tori separation parameter} $\epsilon_i$ which is set to be $1/2$ if any of the components of $X(\alpha)$ bounded by $\alpha_i$ is a torus with one boundary component and $1$ in any other case. \\

The following consequence of Theorem \ref{theo:mw_quant_rec} will allow us control the measure along the fibers of the cut and glue fibration.\\

\begin{proposition}
	\label{prop:twist_tori}
	For every $\delta > 0$ there exists $\epsilon > 0$ such that 
	if $(X,\alpha) \in \mathcal{T}_{g,n}/\text{Stab}_0(\gamma)$ is a point on the fiber of the cut and glue fibration
	\[
	\Psi \colon \mathcal{T}_{g,n}/\text{Stab}_0(\gamma) \to \Omega_{g,n}(\gamma)
	\]
	above a generic pair $((\ell_i)_{i=1}^k,(X(\alpha)_j)_{j=1}^c) \in \Omega_{g,n}(\gamma)$ satisfying
	\begin{enumerate}
		\item $\ell_i \geq \epsilon$ for all $i \in \{1,\dots,k\}$,
		\item $X(\alpha)_j \in \mathcal{M}_{g_j,n_j}^{b_j}(\mathbf{L}_j)_\epsilon$ for all $j \in \{1,\dots,c\}$,
	\end{enumerate}
	then
	\[
	\text{Leb}\left(\left\lbrace (t_1,\dots,t_k) \in \prod_{i=1}^k \left( \mathbf{R}/\epsilon_i \ell_i \mathbf{Z} \right) \ \bigg\vert \ \text{tw}_{\alpha_1}^{t_1} \circ \dots \circ \text{tw}_{\alpha_k}^{t_k}(X) \notin K_\epsilon\right\rbrace\right) < \delta \cdot \prod_{i=1}^k \epsilon_i \ell_i,
	\]
	where $\text{Leb}$ denotes the standard Lebesgue measure on the torus $\prod_{i=1}^k \left( \mathbf{R}/\epsilon_i \ell_i \mathbf{Z} \right)$.
\end{proposition}
$ $

\begin{proof}
	Let $\delta > 0$ be arbitrary and $\epsilon > 0$ be as in Theorem \ref{theo:mw_quant_rec}. Consider a point $(X,\alpha) \in \mathcal{T}_{g,n}/\text{Stab}_0(\gamma)$ satisfying the conditions in the statement of Proposition \ref{prop:twist_tori}. We can parametrize the tori of all twists of $X$ along the components of $\alpha$ by considering the map
	\[
	(t_1,\dots,t_k) \in \prod_{i=1}^k \left( \mathbf{R}/\epsilon_i \ell_i \mathbf{Z} \right) \mapsto \text{tw}_{\alpha_1}^{t_1} \circ \dots \circ \text{tw}_{\alpha_k}^{t_k}(X).
	\]
	It is convenient to consider the linear change of coordinates
	\[
	(t_1,\dots,t_k) := u_1 \cdot (\epsilon_1 \ell_1, \dots, \epsilon_k \ell_k) 
	+ u_2 \cdot (0,\epsilon_2\ell_2,0,\dots,0) + \cdots + u_k \cdot (0,\dots,0,\epsilon_k\ell_k),
	\]
	with $(u_1,\dots,u_k) \in (\mathbf{R}/\mathbf{Z})^k$. The Jacobian $J := \partial  t/\partial u$ of this change of coordinates is given by
	\[
	J = \prod_{i=1}^k \epsilon_i \ell_i.
	\]
	$ $
	
	Notice that the closed path
	\[
	u_1 \in (\mathbf{R}/\mathbf{Z}) \mapsto (u_1,0,\dots,0)
	\]
	in the $(u_1,\dots,u_k)$ coordinates corresponds to the closed path 
	\[
	u_1 \in (\mathbf{R}/\mathbf{Z}) \mapsto \text{tw}_{\lambda_1}^{u_1}(X)
	\]
	in $\mathcal{M}_{g,n}$ given by twisting $X$ along the unordered, weighted, simple closed multi-curve
	\[
	\lambda_1 := \epsilon_1 \ell_1 \alpha_1 + \cdots + \epsilon_k \ell_k \alpha_k \in \mathcal{ML}_{g,n}.
	\] 
	The conditions in the statement of Proposition \ref{prop:twist_tori} precisely ensure that	the obstruction given by condition (2) in Theorem \ref{theo:mw_quant_rec} does not hold for the pair $(X,\lambda_1/(\epsilon_1\ell_1^2 + \cdots + \epsilon_k \ell_k^2)) \in P^1\mathcal{M}_{g,n}$. It follows from Theorem \ref{theo:mw_quant_rec} that
	\[
	\text{Leb}(\{u_1 \in (\mathbf{R}/\mathbf{Z})  \ | \ \text{tw}_{\lambda_1}^{u_1}(X) \in K_\epsilon\}) \geq 1 - \delta.
	\]
	$ $
	
	Notice that, for every $(w_1,\dots,w_k) \in (\mathbf{R}/\mathbf{Z})^k$, the closed path
	\[
	u_2 \in (\mathbf{R}/\mathbf{Z}) \mapsto (w_1,\dots,w_k) + (0,u_2,0,\dots,0)
	\]
	in the $(u_1,\dots,u_k)$ coordinates corresponds to the closed path 
	\[
	u_2 \in (\mathbf{R}/\mathbf{Z}) \mapsto \text{tw}_{\lambda_2}^{u_2}(Y)
	\]
	in $\mathcal{M}_{g,n}$ given by twisting a hyperbolic surface $Y$ on the same $\Psi$-fiber as $X$ along the unordered, weighted, simple closed multi-curve
	\[
	\lambda_2 := \epsilon_2 \ell_2\alpha_2 \in \mathcal{ML}_{g,n}.
	\]  
	For the hyperbolic surfaces $\text{tw}_{\lambda_1}^{u_1}(X) \in K_\epsilon$ the obstruction given by condition (2) in Theorem \ref{theo:mw_quant_rec} does not hold for the pair $(\text{tw}_{\lambda_1}^{u_1}(X),\lambda_2/\epsilon_2\ell_2^2) \in P^1 \mathcal{ML}_{g,n}$ because by the definition of $K_\epsilon$ such hyperbolic surfaces have no simple closed geodesics of length $< \epsilon$. Theorem \ref{theo:mw_quant_rec} together with Fubini's theorem imply
	\[
	\text{Leb}(\{ (u_1,u_2) \in (\mathbf{R}/\mathbf{Z})^2 \ | \ \text{tw}_\lambda^{u_1}(X) \circ \text{tw}_{\lambda_2}^{u_2}(X) \in K_\epsilon\}) \geq (1 - \delta)^2.
	\]
	$ $
	
	For every $i =3,\dots,k$ consider the unordered, weighted, simple closed multi-curve
	\[
	\lambda_i := \epsilon_i \ell_i \alpha_i \in \mathcal{ML}_{g,n}.
	\]
	Proceeding inductively  using the same ideas as above we deduce
	\[
	\text{Leb}\left(\left\lbrace (u_1,\dots,u_k) \in (\mathbf{R}/ \mathbf{Z})^k \ \bigg\vert \ \text{tw}_{\lambda_1}^{u_1} \circ \dots \circ \text{tw}_{\lambda_k}^{u_k}(X) \in K_\epsilon\right\rbrace\right) \geq (1-\delta)^k.
	\]
	Taking complements it follows that
	\[
	\text{Leb}\left(\left\lbrace (u_1,\dots,u_k) \in (\mathbf{R}/ \mathbf{Z})^k \ \bigg\vert \ \text{tw}_{\lambda_1}^{u_1} \circ \dots \circ \text{tw}_{\lambda_k}^{u_k}(X) \notin K_\epsilon\right\rbrace\right) < 1 - (1-\delta)^k.
	\]
	Reversing the change of coordinates we get
	\begin{align*}
	&\text{Leb}\left(\left\lbrace (t_1,\dots,t_k) \in \prod_{i=1}^k \left( \mathbf{R}/\epsilon_i \ell_i \mathbf{Z} \right) \ \bigg\vert \ \text{tw}_{\alpha_1}^{t_1} \circ \dots \circ \text{tw}_{\alpha_k}^{t_k}(X) \notin K_\epsilon\right\rbrace\right) \\
	&< \left(1 -(1-\delta)^k\right) \cdot \prod_{i=1}^k \epsilon_i \ell_i.
	\end{align*}
	As $\delta > 0$ is arbitrary, this finishes the proof.
\end{proof}
$ $

To actually make use of Proposition \ref{prop:twist_tori}, we need to control the volume of the set of pairs $((\ell_i)_{i=1}^k,(X(\alpha)_j)_{j=1}^c) \in \Omega_{g,n}(\gamma)$ that do not satisfy the conditions in the statement of Proposition \ref{prop:twist_tori}. This control is achieved using the following estimate, which is a consequence of Bers's Theorem; see Theorem \ref{theo:bers}.\\

\begin{proposition}
	\label{prop:bad_pt_control}
	Let $g',n',b' \in \mathbf{Z}_{\geq 0}$ be a triple of non-negative integers satisfying $2 - 2g' - n' - b' < 0$ and $\epsilon_0 > 0$ be fixed. There exist constants $C>0$ and $L_0 > 0$ such that for every $\mathbf{L}:= (L_i)_{i=1}^{b'} \in (\mathbf{R}_{>0})^{b'}$ satisfying $M(\mathbf{L})\geq L_0$ and every $0 < \epsilon < \epsilon_0$,
	\[
	\text{Vol}_{wp}(\mathcal{M}_{g',n'}^{b'}(\mathbf{L}) \setminus \mathcal{M}_{g',n'}^{b'}(\mathbf{L})_\epsilon)\leq C \cdot \epsilon^2 \cdot M(\mathbf{L})^{6g'-6+2n'+2b'-2}.
	\]
\end{proposition}
$ $

\begin{proof}
	Let $A_{g',n'}^{b'}, B_{g',n'}^{b'} > 0$ be as in Theorem \ref{theo:bers}. Consider an arbitrary $\mathbf{L}:= (L_i)_{i=1}^{b'} \in (\mathbf{R}_{>0})^{b'}$ and an arbitrary $0 < \epsilon < \epsilon_0$. As a consequence of Theorem \ref{theo:bers}, for every
	$X \in \mathcal{M}_{g',n'}^{b'}(\mathbf{L}) \backslash \mathcal{M}_{g',n'}^{b'}(\mathbf{L})_\epsilon$ we can find a geodesic pair of pants decomposition $\{\alpha_j\}_{j=1}^{3g'-3+n'+b'}$ of $X$ satisfying
	\begin{enumerate}
		\item $\ell_{\alpha_1}(X) < \epsilon$,
		\item $\ell_{\alpha_j}(X) \leq A_{g',n'}^{b'} + B_{g',n'}^{b'} \cdot \max\{M(\mathbf{L}),\epsilon\}$ for all $j \in \{ 2,\dots,3g'-3+n'+b'\}$.
	\end{enumerate}
	Let $L_0 := \max\{\epsilon_0, A_{g',n'}^{b'}/B_{g',n'}^{b'}\}$ so that
	\[
	A_{g',n'}^{b'} + B_{g',n'}^{b'} \cdot \max\{L,\epsilon\} \leq  2B_{g',n'}^{b'} L
	\]
	for every $L \geq L_0$. For the rest of this proof we assume $M(\mathbf{L}) \geq L_0$. Notice there are finitely many pair of pants decompositions of $S_{g',n'}^{b'}$ up to the action of $\text{Mod}_{g',n'}^{b'}$; let $\kappa_{g',n'}^{b'} \in \mathbf{N}$ be the number of such equivalence classes. It follows that $\mathcal{M}_{g',n'}^{b'}(\mathbf{L}) \setminus \mathcal{M}_{g',n'}^{b'}(\mathbf{L})_\epsilon$ can be covered by $\kappa_{g',n'}^{b'}$ subsets of $\mathcal{T}_{g',n'}^{b'}(\mathbf{L})$ which in appropriate Fenchel-Nielsen coordinates $(\ell_j,\tau_j)_{j=1}^{3g'-3+n'+b'} \in (\mathbf{R}_{>0} \times \mathbf{R})^{3g'-3+n'+b'}$ correspond to the set $	\mathcal{A}_{g',n',b'}^{\epsilon,M(\mathbf{L})}$ given by
	\[
	\left\lbrace
	\begin{array}{l | l}
	
	(\ell_j,\tau_j)_{j=1}^{3g'-3+n'+b'} 
	& \ 0 \leq \tau_j < \ell_j, \ \forall j=1,\dots,3g'-3+n'+b',\\
	& \ 0 <\ell_1 \leq \epsilon,\\
	& \ 0 < \ell_j \leq 2B_{g',n'}^{b'}M(\mathbf{L}), \ \forall j=2,\dots,3g'-3+n'+b'.\\
	\end{array} \right\rbrace.
	\]
	Notice that
	\[
	\text{Leb}\left(\mathcal{A}_{g',n',b'}^{\epsilon,M(\mathbf{L})}\right) = \frac{\epsilon^2}{2} \cdot \frac{(2B_{g',n'}^{b'}M(\mathbf{L}))^{6g'-6+2n'+2b'-2}}{2^{3g'-3+n'+b'-1}},
	\]
	where $\text{Leb}$ denotes the standard Lebesgue measure on $(\mathbf{R}_{>0} \times \mathbf{R})^{3g'-3+n'+b'}$. By Wolpert's magic formula,
	\[
	\text{Vol}_{wp}(\mathcal{M}_{g',n'}^{b'}(\mathbf{L}) \setminus \mathcal{M}_{g',n'}^{b'}(\mathbf{L})_\epsilon)\leq  C \cdot \epsilon^2 \cdot M(\mathbf{L})^{6g'-6+2n'+2b'-2},
	\]
	where
	\[
	C := \kappa_{g',n'}^{b'}\cdot 2^{3g'-3+n'+b'-2} \cdot (B_{g',n'}^{b'})^{6g'-6+2n'+2b'-2}.
	\]
	This finishes the proof.
\end{proof}
$ $

We are now ready to prove Proposition \ref{prop:hor_meas_qnem}.\\

\begin{proof}[Proof of Proposition \ref{prop:hor_meas_qnem}]
	It remains to show that for every $\delta > 0$ there exists $\epsilon > 0$ such that
	\[
	\liminf_{L \to \infty} \frac{\widehat{\mu}_\gamma^{a,L}( K_\epsilon)}{m_\gamma^{a,L}} \geq 1 - \delta.
	\]
	We prove the complementary statement: for every $\delta > 0$ there exists $\epsilon > 0$ such that
	\[
	\limsup_{L \to \infty} \frac{\widehat{\mu}_\gamma^{a,L}( \mathcal{M}_{g,n} \setminus K_\epsilon)}{m_\gamma^{a,L}} \leq \delta.
	\]
	$ $
	
	Let $\delta > 0$ be arbitrary and $\epsilon > 0$ be as in Proposition \ref{prop:twist_tori}. Denote by $\widetilde{K_\epsilon} \subseteq \mathcal{T}_{g,n}/\text{Stab}(\gamma)$ and $\dot{K_\epsilon} \subseteq \mathcal{T}_{g,n}/\text{Stab}_0(\gamma)$ the subsets covering $K_\epsilon \subseteq \mathcal{M}_{g,n}$ under the corresponding quotient maps. As the measure $\widehat{\mu}_\gamma^{a,L}$ on $\mathcal{M}_{g,n}$ is the pushforward of the measure $\widetilde{\mu}_\gamma^{a,L}$ on $\mathcal{T}_{g,n}/\text{Stab}(\gamma)$,
	\[
	\widehat{\mu}_\gamma^{a,L}(\mathcal{M}_{g,n} \setminus K_\epsilon) = \widetilde{\mu}_\gamma^{a,L}\left((\mathcal{T}_{g,n}/\text{Stab}(\gamma)) \setminus \widetilde{K_\epsilon}\right).
	\]
	Let $\dot{\mu}_{\gamma}^{a,L}$ be the local pushforward of the measure $\mu_{\gamma}^{a,L}$ on $\mathcal{T}_{g,n}$ to the quotient $\mathcal{T}_{g,n}/\text{Stab}_0(\gamma)$. As $[\text{Stab}(\gamma) : \text{Stab}_0(\gamma)] < \infty$,
	\[
	\dot{\mu}_\gamma^{a,L}\left((\mathcal{T}_{g,n}/\text{Stab}_0(\gamma)) \setminus \dot{K_\epsilon}\right) = 	[\text{Stab}(\gamma) : \text{Stab}_0(\gamma)]  \cdot \widetilde{\mu}_\gamma^{a,L}\left((\mathcal{T}_{g,n}/\text{Stab}(\gamma)) \setminus \widetilde{K_\epsilon}\right). 
	\]
	It is enough then to estimate
	\[
	\frac{\dot{\mu}_\gamma^{a,L}\left((\mathcal{T}_{g,n}/\text{Stab}_0(\gamma)) \setminus \dot{K_\epsilon}\right)}{m_\gamma^{a,L}}.
	\]
	$ $
	
	To do so we use the cut and glue fibration. As we only need an upper bound, we disregard the constant factors (depending only on $g$, $n$, and $\gamma$) in the statement of Theorem \ref{theo:cut_glue_fib}. We also assume that $L \geq 1$. Notice that a point $(X,\alpha) \in \mathcal{T}_{g,n}/\text{Stab}_0(\gamma)$ is not in $\dot{K_\epsilon}$ if and only if $X$ has a simple closed geodesic $\beta$ of length $< \epsilon$. Such a geodesic $\beta$ can show up in one of three ways:
	\begin{enumerate}
		\item $\beta$ is one of the components of $\alpha$.
		\item $\beta$ is disjoint from every component of $\alpha$.
		\item $\beta$ intersects one of the components of $\alpha$.
	\end{enumerate}
	We can thus bound
	\[
	\dot{\mu}_\gamma^{a,L}\left((\mathcal{T}_{g,n}/\text{Stab}_0(\gamma)) \setminus \dot{K_\epsilon}\right) \leq \dot{\mu}_\gamma^{a,L}(A_1) + \dot{\mu}_\gamma^{a,L}(A_2) + \dot{\mu}_\gamma^{a,L}(A_3),
	\]
	where 
	\begin{enumerate}
		\item $A_1 \subseteq \mathcal{T}_{g,n}/\text{Stab}_0(\gamma)$ is the set of all pairs $(X,\alpha) \in \mathcal{T}_{g,n}/\text{Stab}_0(\gamma)$ such that $\ell_{\alpha_i}(X) < \epsilon$ for some $i \in \{1,\dots,k\}$.
		\item $A_2 \subseteq \mathcal{T}_{g,n}/\text{Stab}_0(\gamma)$ is the set of all pairs $(X,\alpha) \in \mathcal{T}_{g,n}/\text{Stab}_0(\gamma)$ such that $\ell_{\beta}(X) < \epsilon$ for some simple closed geodesic $\beta$ on $X$ disjoint from every component of $\alpha$.
		\item $A_3 \subseteq \mathcal{T}_{g,n}/\text{Stab}_0(\gamma)$ is the set of all pairs $(X,\alpha) \in \mathcal{T}_{g,n}/\text{Stab}_0(\gamma)$ satisfying the conditions of Proposition \ref{prop:twist_tori} and such that $(X,\alpha) \notin \dot{K_\epsilon}$.\\
	\end{enumerate}
	
	We first estimate $\dot{\mu}_\gamma^{a,L}(A_1)$. Notice that
	\[
	\dot{\mu}_\gamma^{a,L}(A_1) \leq \dot{\mu}_\gamma^{a,L}(A_1^1) + \cdots + \dot{\mu}_\gamma^{a,L}(A_1^k),
	\]
	where, for every $i = 1, \dots, k$, $A_1^i \subseteq \mathcal{T}_{g,n}/\text{Stab}_0(\gamma)$ is the set of all pairs $(X,\alpha) \in \mathcal{T}_{g,n}/\text{Stab}_0(\gamma)$ such that $\ell_{\alpha_i}(X) < \epsilon$. Let us estimate $\dot{\mu}_\gamma^{a,L}(A_1)$. By Theorem \ref{theo:cut_glue_fib},
	\[
	\dot{\mu}_\gamma^{a,L}(A_1^1)  \leq  \int_0^\epsilon \int_0^{aL} \cdots \int_0^{aL}  \left(\prod_{j=1}^c V_{g_j,n_j}^{b_j}(\mathbf{L}_j)\right) \cdot \ell_1 \cdots \ell_k \ d\ell_1 \cdots d\ell_k.
	\]
	Notice that the integrand
	\[
	\left(\prod_{j=1}^c V_{g_j,n_j}^{b_j}(\mathbf{L}_j)\right) \cdot \ell_1 \cdots \ell_k
	\]
	is a polynomial of degree $6g-6+2n-k$ on the integrating variables $\ell_1,\dots,\ell_k$ all of whose non-zero monomials have positive degree for the $\ell_1$ variable. It follows that
	\[
	\dot{\mu}_\gamma^{a,L}(A_1^1) \leq C_1^1 \cdot \epsilon^2 \cdot L^{6g-6+2n -2},
	\]
	where $C_1^1 > 0$ is a constant depending only on $g$, $n$, $\gamma$, and $a$. Analogously, for every $i = 2, \dots, k$,
	\[
	\dot{\mu}_\gamma^{a,L}(A_1^i) \leq C_1^i \cdot \epsilon^2 \cdot L^{6g-6+2n -2},
	\] 
	where $C_1^i > 0$ is a constant depending only on $g$, $n$, $\gamma$, and $a$. It follows that
	\[
	\dot{\mu}_\gamma^{a,L}(A_1) \leq C_1 \cdot\epsilon^2 \cdot  L^{6g-6+2n -2},
	\]
	where
	\[
	C_1 := C_1^1 + \cdots + C_1^k.
	\]
	$ $
	
	We now estimate $\dot{\mu}_\gamma^{a,L}(A_2)$. Notice that
	\[
	\dot{\mu}_\gamma^{a,L}(A_2) \leq \dot{\mu}_\gamma^{a,L}(A_2^1) + \dots + \dot{\mu}_\gamma^{a,L}(A_2^c),
	\]
	where, for every $j=1,\dots,c$, $A_2^j \subseteq \mathcal{T}_{g,n}/\text{Stab}_0(\gamma)$ is the set of all pairs $(X,\alpha) \in \mathcal{T}_{g,n}/\text{Stab}_0(\gamma)$ such that $\ell_{\beta}(X) < \epsilon$ for some (non-boundary parallel) simple closed geodesic $\beta$ on $X(\alpha)_j$, the $j$-th component of the hyperbolic surface $X(\alpha)$ we get by cutting $X$ along the components of $\alpha$. Let us estimate $\dot{\mu}_\gamma^{a,L}(A_2^1)$. By Theorem \ref{theo:cut_glue_fib}, 
	\[
	\dot{\mu}_\gamma^{a,L}(A_2^1) =  \int_0^{aL} \cdots \int_0^{aL}  V_{g_1,n_1}^{b_1}(\mathbf{L}_1)_\epsilon \cdot \left(\prod_{j=2}^c V_{g_j,n_j}^{b_j}(\mathbf{L}_j)\right) \ \ell_1 \cdots \ell_k \cdot d\ell_1 \cdots d\ell_k,
	\]
	where 
	\[
	V_{g_1,n_1}^{b_1}(\mathbf{L}_1)_\epsilon  := \text{Vol}_{\text{wp}}(\mathcal{M}_{g_1,n_1}^{b_1}(\mathbf{L}_1) \setminus \mathcal{M}_{g_1,n_1}^{b_1}(\mathbf{L}_1)_\epsilon).
	\]
	By Proposition \ref{prop:bad_pt_control},
	\[
	V_{g_1,n_1}^{b_1}(\mathbf{L}_1)_\epsilon \leq C_{g_1,n_1}^{b_1} \cdot \epsilon^2 \cdot M(\mathbf{L}_1)^{6g_1-6 +2n_1 + 2b_1}.
	\]
	It follows that
	\begin{align*}
	&\dot{\mu}_\gamma^{a,L}(A_2^1) \leq  C_{g_1,n_1}^{b_1} \cdot \epsilon^2 \cdot \\
	&\int_0^{aL} \cdots \int_0^{aL} M(\mathbf{L}_1)^{6g_1-6 +2n_1 + 2b_1-2} \cdot \left(\prod_{j=2}^c V_{g_j,n_j}^{b_j}(\mathbf{L}_j)\right) \cdot \ell_1 \cdots \ell_k \ d\ell_1 \cdots d\ell_k.
	\end{align*}
	As 
	\[
	M(\mathbf{L}_1) \leq (\mathbf{L}_1)_1 + \cdots + (\mathbf{L}_1)_{b_1},
	\]
	this integral can be bounded by the sum of $b_1$ integrals of polynomials of degree $6g-6+2n-k-2$ on the integrating variables $\ell_1,\dots,\ell_k$. It follows that
	\[
	\dot{\mu}_\gamma^{a,L}(A_2^1)  \leq C_2^1 \cdot \epsilon^2 \cdot L^{6g-6+2n-2},
	\]
	where $C_2^1 > 0$ is a constant depending only on $g$, $n$, $\gamma$, and $a$. Analogously, for every $j = 2, \dots, c$,
	\[
	\dot{\mu}_\gamma^{a,L}(A_2^j) \leq C_2^j \cdot \epsilon^2 \cdot L^{6g-6+2n-2},
	\]
	where $C_2^j > 0$ is a constant depending only on $g$, $n$, $\gamma$, and $a$. It follows that 
	\[
	\dot{\mu}_\gamma^{a,L}(A_2) \leq C_2 \cdot \epsilon^2 \cdot L^{6g-6+2n-2},
	\]
	where
	\[
	C_2 := C_2^1 + \cdots + C_2^c.
	\]
	$ $
	
	Finally, we estimate $\dot{\mu}_\gamma^{a,L}(A_3)$. By Theorem \ref{theo:cut_glue_fib} and Proposition \ref{prop:twist_tori},
	\[
	\dot{\mu}_\gamma^{a,L}(A_3)\leq \delta \cdot \int_0^{aL} \cdots \int_0^{aL}  \left(\prod_{j=1}^c V_{g_j,n_j}^{b_j}(\mathbf{L}_j)\right) \cdot \ell_1 \cdots \ell_k \ d\ell_1 \cdots d\ell_k.
	\]
	As the integrand
	\[
	\left(\prod_{j=1}^c V_{g_j,n_j}^{b_j}(\mathbf{L}_j)\right) \cdot \ell_1 \cdots \ell_k
	\]
	is a polynomial of degree $6g-6+2n-k$ on the integrating variables $\ell_1,\dots,\ell_k$, 
	\[
	\dot{\mu}_\gamma^{a,L}(A_3)\leq C_3 \cdot \delta \cdot L^{6g-6+2n},
	\]
	where $C_3 > 0$ is a constant depending only on $g$, $n$, $\gamma$, and $a$.\\

	Adding up all the contributions we deduce
	\[
	\dot{\mu}_\gamma^{a,L}\left((\mathcal{T}_{g,n}/\text{Stab}_0(\gamma)) \setminus \dot{K_\epsilon}\right) \leq (C_1+ C_2) \cdot\epsilon^2 \cdot  L^{6g-6+2n -2} + C_3 \cdot \delta \cdot L^{6g-6+2n}.
	\]
	Dividing by total mass $m_\gamma^{a,L}$ we get
	\[
	\frac{\dot{\mu}_\gamma^{a,L}\left((\mathcal{T}_{g,n}/\text{Stab}_0(\gamma)) \setminus \dot{K_\epsilon}\right)}{m_\gamma^{a,L}} \leq (C_1+ C_2) \cdot\epsilon^2 \cdot  \frac{L^{6g-6+2n -2}}{m_\gamma^{a,L}} + C_3 \cdot \delta \cdot \frac{L^{6g-6+2n}}{m_\gamma^{a,L}}.
	\]
	By Proposition \ref{prop:total_hor_meas},
	\[
	\lim_{L \to \infty} \frac{m_\gamma^{a,L}}{L^{6g-6+2n}} = c,
	\]
	where $c > 0$ is a positive constant depending only on $g$, $n$, $\gamma$ and $a$. It follows that
	\[
	\limsup_{L \to \infty} \frac{\dot{\mu}_\gamma^{a,L}((\mathcal{T}_{g,n}/\text{Stab}_0(\gamma)) \setminus \dot{K_\epsilon})}{m_\gamma^{a,L}} \leq \frac{C_3}{c} \cdot \delta.
	\]
	As $\delta > 0$ is arbitrary, this finishes the proof.
\end{proof}
$ $

\section{Equidistribution of horospheres}

$ $

\textit{Setting.} For the rest of this section, fix an ordered simple closed multi-curve $\gamma := (\gamma_1,\dots,\gamma_k)$ on $S_{g,n}$ with $1 \leq k \leq 3g-3+n$, a vector $\mathbf{a}:= (a_1,\dots,a_k) \in (\mathbf{Q}_{>0})^k$ of positive rational weights on the components of $\gamma$, and a bounded, Borel measurable function $f \colon \Delta_a \to \mathbf{R}_{\geq0}$ with non-negative values and which is not almost everywhere zero with respect to the Lebesgue measure class. \\

The purpose of this section is to give a brief overview of the proof of Theorem \ref{theo:horosphere_equid}. We mimic the outline of the proof of Theorem \ref{theo:horoball_equid} in \S 1. In particular, we discuss analogues of Propositions \ref{prop:earth_inv_0}, \ref{prop:hor_meas_ac_0},  \ref{prop:hor_meas_nem_0}, and \ref{prop:total_hor_meas} for horosphere sector measures. As mentioned in \S 1, a completely new idea that needs to be introduced in the proof of Theorem \ref{theo:horosphere_equid} is the Mirzakhani bound; see Proposition \ref{prop:mir_bd_og_0} and Proposition \ref{prop:mir_bd_og} below. This bound is used in the proof of the absolute continuity with respect to the Mirzakhani measure of the limit points of the horosphere sector measures.\\

\textit{Total mass.} Consider $(\mathbf{R}_{\geq 0})^k$ as a smooth manifold with coordinates $x_1,\dots,x_k$ and endow it with the volume form
\[
\omega := dx_1 \wedge \dots \wedge dx_k.
\] 
Let $f_\mathbf{a} \colon (\mathbf{R}_{\geq0})^k \to \mathbf{R}_{\geq0}$ be the function which to every $(x_1,\dots,x_k) \in (\mathbf{R}_{\geq0})^k$ assigns the value
\[
f_\mathbf{a}(x_1,\dots,x_k) := a_1x_1 + \cdots + a_k x_k.
\]
For every $L>0$, the volume form $\omega$ induces a volume form $\omega_\mathbf{a}^L := \omega_{f_\mathbf{a}}^L$ on the level set
$
f_\mathbf{a}^{-1}(L) = \Delta_{L \cdot \mathbf{a}}
$
by contraction; see \S 2. Let $\eta_\mathbf{a} := |\omega_\mathbf{a}^1|$ be the measure induced by the volume form $\omega_\mathbf{a}^1$ on the level set $f_\mathbf{a}^{-1}(1) = \Delta_{\mathbf{a}}$. We also denote by $\eta_\mathbf{a}$ the extension by zero of this measure to all $(\mathbf{R}_{\geq 0})^k$. The following result is an analogue of Proposition \ref{prop:total_hor_meas} for horosphere sector measures; it can be proved using the cut and glue fibration, see Theorem \ref{theo:cut_glue_fib}, in a similar way to Proposition \ref{prop:total_hor_meas}.\\

\begin{proposition}
	\label{prop:total_horosphere_meas}
	For every $L > 0$,
	\[
	n_{\gamma,\mathbf{a}}^{f,L} = \int_{\Delta_\mathbf{a}} f(\mathbf{L}) \cdot V_{g,n}(\gamma,L \cdot \mathbf{L}) \cdot L^{k-1} \  d \eta_{\mathbf{a}}(\mathbf{L}),
	\]
	where $\mathbf{L} := (\ell_1,\dots,\ell_k) \in \Delta_a$. In particular,
	\[
	\lim_{L \to \infty} \frac{n_{\gamma,\mathbf{a}}^{f,L}}{L^{6g-6+2n-1}} = \int_{\Delta_\mathbf{a}} f(\mathbf{L}) \cdot W_{g,n}(\gamma,\mathbf{L}) \ d\eta_{\mathbf{a}}(\mathbf{L}).
	\]
\end{proposition}
$ $

\textit{Earthquake flow invariance.} The following result is an analogue of Proposition \ref{prop:earth_inv_0} for horosphere sector measures.\\

\begin{proposition}
	\label{prop:earth_inv_1}
	Any weak-$\star$ limit point of the sequence of probability measures $\{\widehat{\zeta}_{\gamma,\mathbf{a}}^{f,L}/n_{\gamma,\mathbf{a}}^{f,L}\}_{L > 0}$ on $P^1\mathcal{M}_{g,n}$ is earthquake flow invariant.
\end{proposition}
$ $

Just as in the case of Proposition \ref{prop:earth_inv_0}, Proposition \ref{prop:earth_inv_1} is a  direct consequence of the continuity of the earthquake flow on $P^1\mathcal{M}_{g,n}$ and the following analogue of Proposition \ref{prop:inv_meas} for horosphere sector measures; it can be proved using similar arguments.\\

\begin{proposition}
	\label{prop:inv_meas_sph}
	The measures $\widehat{\zeta}_{\gamma,\mathbf{a}}^{f,L}$ on $P^1 \mathcal{M}_g$ are earthquake flow invariant.
\end{proposition}
$ $

\textit{Comparing horosphere sector measures.} Recall the definition of the horosphere measures $\eta_{\gamma,\mathbf{a}}^L$ on $\mathcal{T}_{g,n}$ introduced in \S 1. These measures induce horosphere measures $\zeta_{\gamma,\mathbf{a}}^L$ on the bundle $P^1\mathcal{T}_{g,n}$ by considering the disintegration formula
\[
d \zeta_{\gamma,\mathbf{a}}^{L}(X,\lambda) := d \delta_{\mathbf{a}\cdot\gamma/ \ell_{\mathbf{a}\cdot\gamma}(X)}(\lambda) \ d\eta_{\gamma,\mathbf{a}}^{L}(X).
\]
Let $\widetilde{\eta}_{\gamma,\mathbf{a}}^L$ be the local pushforward of $\eta_{\gamma,\mathbf{a}}$ to $\mathcal{T}_{g,n}/\text{Stab}(\gamma)$ and $\widehat{\eta}_{\gamma,\mathbf{a}}^L$ be the pushforward of $\widetilde{\eta}_{\gamma,\mathbf{a}}^L$ to $\mathcal{M}_{g,n}$. Let $\widetilde{\zeta}_{\gamma,\mathbf{a}}^L$ be the local pushforward of $\zeta_{\gamma,\mathbf{a}}$ to $P^1\mathcal{T}_{g,n}/\text{Stab}(\gamma)$ and $\widehat{\zeta}_{\gamma,\mathbf{a}}^L$ be the pushforward of $\widetilde{\zeta}_{\gamma,\mathbf{a}}^L$ to $P^1\mathcal{M}_{g,n}$. Denote by $n_{\gamma,\mathbf{a}}^L$ the total mass of the measures $\widehat{\eta}_{\gamma,\mathbf{a}}^L$ and $\widehat{\zeta}_{\gamma,\mathbf{a}}^L$, i.e.,
\[
n_{\gamma,\mathbf{a}}^L = \widehat{\eta}_{\gamma,\mathbf{a}}^L(\mathcal{M}_{g,n}) = \widehat{\zeta}_{\gamma,\mathbf{a}}^L(P^1 \mathcal{M}_{g,n}).
\]
$ $

Much like in the case of Theorem \ref{theo:horoball_equid}, many steps in the proof of Theorem \ref{theo:horosphere_equid} can be reduced to the study of the sequence of  probability measures $\{\widehat{\zeta}_{\gamma,\mathbf{a}}^L/n_{\gamma,\mathbf{a}}^L\}_{L>0}$ on $P^1\mathcal{M}_{g,n}$. The following result is an analogue of Lemma \ref{lemma:hor_meas_comp} for horosphere sector measures; it can be proved using similar arguments. \\

\begin{lemma}
	\label{lemma:hor_meas_comp_1}
	There exists a constant $C > 0$ such that for every Borel measurable subset $A \subseteq P^1\mathcal{M}_{g,n}$,
	\[
	\limsup_{L \to \infty} \frac{\widehat{\zeta}_{\gamma,\mathbf{a}}^{f,L}(A)}{n_{\gamma,\mathbf{a}}^{f,L}} \leq C \cdot \limsup_{L \to \infty} \frac{\widehat{\zeta}_{\gamma,\mathbf{a}}^{L}(A)}{n_{\gamma,\mathbf{a}}^{L}}.
	\]
\end{lemma}
$ $

\textit{Absolute continuity with respect to the Mirzakhani measure.} The following result is an analogue of Proposition \ref{prop:hor_meas_ac_0} for horosphere sector measures.\\

\begin{proposition}
	\label{prop:hor_meas_ac_1}
	Any weak-$\star$ limit point of the sequence of probability measures $\{\widehat{\zeta}_{\gamma,\mathbf{a}}^{f,L}/n_{\gamma,\mathbf{a}}^{f,L}\}_{L > 0}$ on $P^1\mathcal{M}_{g,n}$ is absolutely continuous with respect to $\widehat{\nu}_{\text{Mir}}$.
\end{proposition}
$ $

To prove Proposition \ref{prop:hor_meas_ac_1} we mimic the arguments in the proof of Proposition \ref{prop:hor_meas_ac_0}; see \S 3. After applying Lemma \ref{lemma:approx}, the proof reduces to the following analogue of Proposition \ref{prop:ac_key_estimate}. \\

\begin{proposition}
	\label{prop:ac_key_estimate_1}
	Let $K \subseteq \mathcal{T}_{g,n}$ be a compact subset. There exist  constants $C > 0$ and $\epsilon_0 > 0$ such that for every $X \in K$, every $0 < \epsilon < \epsilon_0$, and every $V \subseteq P \mathcal{ML}_{g,n}$ open continuity subset of the Lebesgue measure class,
	\[
	\limsup_{L \to \infty} \frac{\widehat{\zeta}_{\gamma,\mathbf{a}}^{f,L}(\Pi(U_X(\epsilon) \times V))}{n_{\gamma,\mathbf{a}}^{f,L}} \leq C \cdot \nu_{\text{Mir}}(U_X(\epsilon) \times V).
	\]
\end{proposition}
$ $

To prove Proposition \ref{prop:ac_key_estimate_1} we mimic the three step proof of Proposition \ref{prop:ac_key_estimate}; significant changes need to be introduced and will be discussed below. By Lemma \ref{lemma:hor_meas_comp_1}, it is enough to prove Proposition \ref{prop:ac_key_estimate_1} for the sequence of  probability measures $\{\widehat{\zeta}_{\gamma,\mathbf{a}}^L/n_{\gamma,\mathbf{a}}^L\}_{L>0}$ on $P^1\mathcal{M}_{g,n}$. It will be convenient to consider $\epsilon_0 > 0$ small enough so that for every $0 < \epsilon < \epsilon_0$,
\[
1-2\epsilon < e^{-\epsilon} < e^\epsilon < 1+2\epsilon.
\]
$ $

Let $X \in \mathcal{T}_{g,n}$, $\epsilon>0$, and $V \subseteq P\mathcal{ML}_{g,n}$ be a Borel measurable subset. For every $L > 0$ consider the counting function
\[
s(X,\mathbf{a}\cdot\gamma,L,\epsilon,V) := \# \{\alpha \in \text{Mod}_{g,n} \cdot (\mathbf{a} \cdot \gamma) \ | \ e^{-\epsilon}L \leq \ell_{\alpha}(X) \leq e^\epsilon L, \ [\alpha] \in V \},
\]
where $\mathbf{a} \cdot \gamma \in \mathcal{ML}_{g,n}(\mathbf{Q})$ is as in (\ref{eq:weight_curve}). The first step in the proof of Proposition \ref{prop:ac_key_estimate_1} corresponds to the following analogue of Lemma \ref{lemma:quot_bound} for horosphere sector measures. The proof of this analogue uses Proposition \ref{prop:mir_bd_og_0}, the Mirzakhani bound, in a crucial way. Proposition \ref{prop:mir_bd_og_0} will be proved later in this section.\\

\begin{lemma}
	\label{lemma:quot_bound_1}
	Let $K \subseteq \mathcal{T}_{g,n}$ be a compact subset. There exist constants $C > 0$ and $\epsilon_0 > 0$ such that for every $X \in K$, every $0 < \epsilon < \epsilon_0$, every Borel measurable subset $V \subseteq P\mathcal{ML}_{g,n}$, and every $L > 0$, 
	\[
	\widehat{\zeta}_{\gamma,\mathbf{a}}^{L}(\Pi(U_X(\epsilon) \times V)) \leq C \cdot s(X,\mathbf{a} \cdot \gamma,L,\epsilon,V) \cdot \frac{\epsilon^{6g-6+2n-1}}{L}.
	\]
\end{lemma}
$ $

\begin{proof}
	Let $C_1 > 0$ and $\epsilon_0 >0$ be the constants provided by Proposition \ref{prop:mir_bd_og_0}. Fix $X \in \mathcal{T}_{g,n}$, $0 < \epsilon < \epsilon_0$, $V \subseteq P\mathcal{ML}_{g,n}$ Borel measurable subset, and $L > 0$. Mimicking the proof of Lemma \ref{lemma:quot_bound} one can show that
	\begin{equation}
	\label{eq:measure_sum_1}
	\widehat{\zeta}_{\gamma,\mathbf{a}}^L(\Pi(U_X(\epsilon) \times V)) \leq \sum_{[\phi] \in \text{Mod}_{g,n}/\text{Stab}(\gamma)} \zeta_{\phi \cdot \gamma,\mathbf{a}}^{L}( U_X(\epsilon) \times V).
	\end{equation}
	Notice that if $\zeta_{\phi \cdot \gamma,\mathbf{a}}^{L}(U_X(\epsilon) \times V) \neq 0$ then $[\phi \cdot (\mathbf{a} \cdot \gamma)] \in V$ and $S_{\phi \cdot \gamma,\mathbf{a}}^{L} \cap U_X(\epsilon) \neq \emptyset$. By the definition of the horoball $S_{\phi \cdot \gamma,\mathbf{a}}^{L}$ and of the symmetric Thurston metric $d_{\text{Thu}}$, the second condition implies
	\[
	e^{-\epsilon} L \leq \ell_{\phi\cdot (\mathbf{a} \cdot \gamma)}(X) \leq e^\epsilon L.
	\]
	It follows that in the sum on the right hand side of (\ref{eq:measure_sum_1}) at most 
	\[[\text{Stab}(\mathbf{a} \cdot \gamma):\text{Stab}(\gamma)] \cdot s(X,\mathbf{a}\cdot\gamma,L,\epsilon,V)\] 
	summands are non-zero. For each one of these summands, Proposition \ref{prop:mir_bd_og_0} ensures 
	\[
	\zeta_{\phi \cdot \gamma,\mathbf{a}}^{L}( U_X(\epsilon) \times V) \leq \eta_{\phi \cdot \gamma,\mathbf{a}}^L(U_X(\epsilon)) \leq C_1 \cdot \frac{\epsilon^{6g-6g+2n-1}}{L}.
	\]
	Putting things together we deduce 
	\[
	\widehat{\zeta}_{\gamma,\mathbf{a}}^{a,L}(\Pi(U_X(\epsilon) \times V)) \leq C \cdot s(X,\mathbf{a} \cdot \gamma,L,\epsilon,V) \cdot \frac{\epsilon^{6g-6+2n-1}}{L},
	\]
	where 
	\[
	C:= [\text{Stab}(\mathbf{a} \cdot \gamma):\text{Stab}(\gamma)] \cdot C_1.
	\]
	This finishes the proof.
\end{proof}
$ $

The following result, corresponding to the second step in the proof of Proposition \ref{prop:ac_key_estimate_1}, is an analogue of Lemma \ref{lemma:count_bd_ac_1}; it can be proved using similar arguments.\\

\begin{lemma}
	\label{lemma:count_bd_ac_1}
	There exist constants $C > 0$ and $\epsilon_0 > 0$ such that for every $X \in \mathcal{T}_{g,n}$, every $0 < \epsilon < \epsilon_0$, and every $V \subseteq P \mathcal{ML}_{g,n}$ open continuity subset of the Lebesgue measure class,
	\[
	\limsup_{L \to \infty} \frac{s(X,\mathbf{a} \cdot \gamma,L,\epsilon,V)}{L^{6g-6+2n}} \leq C \cdot \epsilon \cdot B_V(X).
	\]
\end{lemma}
$ $

The following result, corresponding to the third step in the proof of Proposition \ref{prop:ac_key_estimate_1}, is an analogue of Lemma \ref{lemma:bv_measure_comparison}. It is a direct consequence of Lemma \ref{lemma:bv_measure_comparison} and Lemma \ref{lemma:thu_ball_bound}.\\

\begin{lemma}
	\label{lemma:bv_measure_comparison_1}
	Let $K \subseteq \mathcal{T}_{g,n}$ be a compact subset. There exist constants $C > 0$ and $\epsilon_0 > 0$ such that for every $X \in K$, every $0 < \epsilon < \epsilon_0$, and every Borel measurable subset $V \subseteq P \mathcal{ML}_{g,n}$, 
	\[
	\epsilon^{6g-6+2n} \cdot B_V(X) \leq C \cdot \nu_{\text{Mir}}(U_X(\epsilon) \times V).
	\]
\end{lemma}
$ $

Proposition \ref{prop:ac_key_estimate_1} now follows directly from Lemmas \ref{lemma:quot_bound_1}, \ref{lemma:count_bd_ac_1}, and \ref{lemma:bv_measure_comparison_1}, and Proposition \ref{prop:total_horosphere_meas}.\\

\textit{No escape of mass}. The following result is an analogue of Proposition \ref{prop:hor_meas_nem_0} for horosphere sector measures.\\

\begin{proposition}
	\label{prop:hor_meas_nem_1}
	Any weak-$\star$ limit point of the sequence of probability measures $\{\widehat{\zeta}_{\gamma,\mathbf{a}}^{f,L}/n_{\gamma,\mathbf{a}}^{f,L}\}_{L > 0}$ on $P^1\mathcal{M}_{g,n}$ is a probability measure.
\end{proposition}
$ $

Proposition \ref{prop:hor_meas_nem_1} can be proved in a similar way to Proposition \ref{prop:hor_meas_nem_0}. Theorem \ref{theo:mw_quant_rec}, its consequence Proposition \ref{prop:twist_tori}, and Theorem \ref{theo:cut_glue_fib} play an important role in the proof.\\

\textit{The Mirzakhani bound.} To finish the proof of Theorem \ref{theo:horosphere_equid}, it remains to prove Proposition \ref{prop:mir_bd_og_0}, the Mirzakhani bound. We will actually prove a stronger result. Recall the definition of the horosphere measures $\eta_{\lambda}^L$ on $\mathcal{T}_{g,n}$ for arbitrary measured geodesic laminations $\lambda \in \mathcal{ML}_{g,n}$ introduced in \S 2.\\

\begin{proposition}
	\label{prop:mir_bd_og}
	Let $K \subseteq \mathcal{T}_{g,n}$ be a compact subset. There exist constants $C > 0$ and $\epsilon_0 > 0$ such that for every $X \in K$, every $0 < \epsilon < \epsilon_0$, every $\lambda \in \mathcal{ML}_{g,n}$, and every $L >0$,
	\[
	\eta_{\lambda}^L(U_X(\epsilon)) \leq C \cdot \frac{\epsilon^{6g-6+2n-1}}{L}.
	\]
\end{proposition}
$ $

We begin by reformulating Proposition \ref{prop:mir_bd_og} to incorporate the dependence in $L$ in a more natural way. For the rest of this discussion, $\omega := v_{\text{wp}}$ will denote the Weil-Petersson volume form on $\mathcal{T}_{g,n}$. Let $\lambda \in \mathcal{ML}_{g,n}$ and $L > 0$ be arbitrary. Notice that, although as sets the horospheres $S_\lambda^L$ and $S_{\lambda/L}^1$ are equal, they carry different horosphere volume forms $\omega_\lambda^L := \omega_{\ell_\lambda}^L$ and $\omega_{\lambda/L}^1 := \omega_{\ell_{\lambda/L}}^1$, and different horosphere measures $\eta_\lambda^L := \eta_{\ell_\lambda}^L$ and $\eta_{\lambda/L}^1 := \eta_{\ell_{\lambda/L}}^1$. The following lemma describes the relation between these volume forms and measures.\\

\begin{lemma}
	\label{lem:vol_form_rel}
	For every $\lambda \in \mathcal{ML}_{g,n}$ and every $L > 0$, 
	\[
	\omega_{\lambda/L}^1 = L \cdot \omega_{\lambda}^L \quad, \quad
	\eta_{\lambda/L}^1 = L \cdot \eta_{\lambda}^L.
	\]
\end{lemma}
$ $

\begin{proof}
	Let $F \colon \mathcal{T}_{g,n} \to T \mathcal{T}_{g,n}$ be a vector field on $\mathcal{T}_{g,n}$ satisfying $d \ell_{\lambda}(F) \equiv 1$. By definition, $\omega_{\lambda}^L := (\iota_{F}\omega)|_{S_{\lambda}^L}$. Notice that $L \cdot F \colon \mathcal{T}_{g,n}\to T \mathcal{T}_{g,n}$ is a vector field on $\mathcal{T}_{g,n}$ satisfying $d \ell_{\lambda/L}(L \cdot F) \equiv 1$. By definition, $\omega_{\lambda/L}^1 := (\iota_{L\cdot F}\omega)|_{S_{\lambda/L}^1}$. It follows that
	\[
	\omega_{\lambda/L}^1 = (\iota_{L\cdot F}\omega)|_{S_{\lambda/L}^1}  = L \cdot (\iota_F\omega)|_{S_{\lambda}^L} = L \cdot \omega_{\lambda}^L.
	\]
	In particular,
	\[
	\eta_{\lambda/L}^1  = \big|\omega_{\lambda/L}^1\big| =  L \cdot \big|\omega_{\lambda}^L\big| = L \cdot \eta_{\lambda}^L.
	\]
	This finishes the proof.
\end{proof}
$ $

It follows directly from Lemma \ref{lem:vol_form_rel} that  Proposition \ref{prop:mir_bd_og} admits the following equivalent reformulation, which gets rid of the dependence in $L$ in the original statement.\\

\begin{proposition}
	\label{prop:mir_bd_ref}
	Let $K \subseteq \mathcal{T}_{g,n}$ be a compact subset. There exist constants $C > 0$ and $\epsilon_0 > 0$ such that for every $X \in K$, every $0 < \epsilon < \epsilon_0$, and every $\lambda \in \mathcal{ML}_{g,n}$, 
	\[
	\eta_{\lambda}^1(U_X(\epsilon)) \leq C \cdot \epsilon^{6g-6+2n-1}.
	\]
\end{proposition}
$ $

The proof of Proposition \ref{prop:mir_bd_ref} reduces to the study of a compact subset of $\mathcal{ML}_{g,n}$, as we now explain. Let $K \subseteq \mathcal{T}_{g,n}$ be a compact subset and $\epsilon_0 > 0$ be fixed. Consider the compact neighborhood $K(\epsilon_0) \subseteq \mathcal{T}_{g,n}$ of $K$ given by all points in $\mathcal{T}_{g,n}$ which are at distance $\leq \epsilon_0$ from $K$ with respect to $d_\text{Thu}$. Consider also the subset $\mathcal{ML}_{g,n}(K(\epsilon_0)) \subseteq \mathcal{ML}_{g,n}$ given by
\[
\mathcal{ML}_{g,n}(K(\epsilon_0)) := \{\lambda \in \mathcal{ML}_{g,n} \colon S_\lambda^1 \cap K(\epsilon_0) \neq \emptyset\}.
\]
Notice that for every $X \in K$, every $0 < \epsilon < \epsilon_0$, and every $\lambda \in \mathcal{ML}_{g,n} \setminus \mathcal{ML}_{g,n}(K(\epsilon_0))$, 
\[
\eta_{\lambda}^1(U_X(\epsilon)) = 0.
\]
In particular, Proposition \ref{prop:mir_bd_ref} holds trivially for such measured geodesic laminations. It remains to prove the desired bound for $\lambda \in \mathcal{ML}_{g,n}(K(\epsilon_0))$. Conveniently enough, this set is compact, as the following lemma shows. This lemma is a consequence of the continuity of the hyperbolic length function $\ell \colon \mathcal{ML}_{g,n} \times \mathcal{T}_{g,n} \to \mathbf{R}_{>0}$ and of the compactness of the space $P\mathcal{ML}_{g,n}$ of projective measured geodesic laminations on $S_{g,n}$.\\

\begin{lemma}
	\label{lem:compact_ML}
	Let $K \subseteq \mathcal{T}_{g,n}$ be a compact subset. Then, the set
	\[
	\mathcal{ML}_{g,n}(K) := \{\lambda \in \mathcal{ML}_{g,n} \colon S_\lambda^1 \cap K \neq \emptyset \}
	\]
	is a compact subset of $\mathcal{ML}_{g,n}$.\\
\end{lemma}

We are now ready to prove Proposition \ref{prop:mir_bd_ref}. Theorem \ref{theo:smooth_length} will play an important role in the proof. \\

\begin{proof}[Proof of Proposition \ref{prop:mir_bd_ref}]
	It remains to show there exist constants $C> 0$ and $\epsilon_0 > 0$ such that for every $X \in K$, every $ 0 < \epsilon < \epsilon_0$, and every $\lambda \in \mathcal{ML}_{g,n}(K(\epsilon_0))$,
	\[
	\eta_{\lambda}^1(U_X(\epsilon)) \leq C \cdot \epsilon^{6g-6+2n-1}.
	\]
	$ $
	
	Let $C_1 > 0$ and $\epsilon_0 > 0$ be the constants provided by Lemma \ref{lemma:thu_ball_bound}. Then, 
	\[
	\mu_{\text{wp}}(U_X(\epsilon)) \leq C_1 \cdot \epsilon^{6g-6+2n}
	\]
	for every $X \in K$ and every $0 < \epsilon < \epsilon_0$. In particular, 
	\begin{equation}
	\label{eq:ball_meas_bound}
	\mu_{\text{wp}}(U_X(2\epsilon)) \leq C_1 \cdot 2^{6g-6+2n} \cdot \epsilon^{6g-6+2n},
	\end{equation}
	for every $X \in K$ and every $0 < \epsilon < \epsilon_0/2$.\\
	
	Fix a complete Riemannian metric $\langle\cdot,\cdot\rangle$ on $\mathcal{T}_{g,n}$ and a smooth non-negative function $\phi \colon \mathcal{T}_{g,n} \to \mathbf{R}_{\geq 0}$ such that $\phi \equiv 1$ on $K(\epsilon_0)$ and $\phi \equiv 0$ outside of $K(2\epsilon_0)$. Let $X \in K$, $0 < \epsilon < \epsilon_0/2$, and $\lambda \in \mathcal{ML}_{g,n}(K(\epsilon_0/2))$ be such that $S_\lambda^1 \cap U_X(\epsilon) \neq \emptyset$. On $\mathcal{T}_{g,n}$ consider the vector field 
	\[
	F^{\lambda} := \phi \cdot \frac{\nabla \ell_{\lambda}}{\langle\nabla \ell_{\lambda},\nabla \ell_{\lambda}\rangle}.
	\]
	Notice that $d \ell_{\lambda}(F^\lambda) \equiv 1$ on $K(\epsilon_0)$ and that $F^\lambda$ vanishes outside of $K(2\epsilon_0)$. In the terminology of \S 2, $(K(\epsilon_0), K(2 \epsilon_0), F^\lambda)$ is a flow datum of $\ell_{\lambda}$ on $\mathcal{T}_{g,n}$. Let $\{\varphi_t^\lambda \colon \mathcal{T}_{g,n} \to \mathcal{T}_{g,n}\}_{t \in \mathbf{R}}$ be the one-parameter group of diffeomorphisms induced by $F^\lambda$ on $\mathcal{T}_{g,n}$.\\
	
	Let $\|\cdot\|$ be the continuous family of fiberwise asymmetric norms on $T\mathcal{T}_{g,n}$ inducing the (Finsler) asymmetric Thurston metric $d_\text{Thu}'$ on $\mathcal{T}_{g,n}$. Consider 
	\[
	C_2(\lambda) := \max \left\lbrace \sup_{Z \in K(2\epsilon_0)} \|\pm F^{\lambda}_Z\| \ , \ 1 \right \rbrace < +\infty.
	\]
	Lemma \ref{lemma:len_var} ensures
	\[
	d_{\text{Thu}}(Y,\varphi_\lambda^t Y) \leq C_2(\lambda) \cdot t
	\]
	for every $Y \in K(2\epsilon_0)$ and every $t \geq 0$. In particular, by the triangle inequality,
	\begin{equation}
	\label{eq:unit_speed_2}
	\varphi^\lambda_t(Y) \in U_X(2\epsilon) \subseteq K(\epsilon_0)
	\end{equation}
	for every $Y \in U_X(\epsilon)$ and every $0 \leq t \leq \epsilon/C_2(\lambda) \leq \epsilon_0/2$.\\

	Consider the set $V := S_\lambda^1 \cap U_X(\epsilon) \neq \emptyset$. Our goal is to establish an appropriate upper bound for $\eta_{\lambda}^1(V)$. Notice that, by (\ref{eq:unit_speed_2}), $\varphi^\lambda_t(V) \subseteq K(\epsilon_0)$ for every $0 \leq t \leq \epsilon/C_2(\lambda)$. In the terminology of \S 2, $(1,V,[0,\epsilon/C_2(\lambda)])$ is a unit speed triple with respect to the flow datum $(K(\epsilon_0), K(2 \epsilon_0), F_\lambda)$. By Proposition \ref{prop:vol_var}, 
	\[
	\eta_\lambda^{1+t} (\varphi^\lambda_t(V)) = (|\text{det}(\varphi^\lambda_t)| \cdot \eta_{\lambda}^1) (V)
	\]
	for every $0 \leq t \leq \epsilon/C_2(\lambda)$. In particular,
	\begin{equation}
	\label{eq:slice_measure_bound}
	\eta_\lambda^{1+t} (\varphi^\lambda_t(V)) \geq C_3(\lambda) \cdot \eta_{\lambda}^1 (V)
	\end{equation}
	for every $0 \leq t \leq \epsilon/C_2(\lambda)$, where 
	\[
	C_3(\lambda) := \inf_{\substack{Z \in K(\epsilon_0),\\ t \in [0,\epsilon_0/2]}} |\text{det}(\varphi^\lambda_t)(Z)| > 0.
	\]
	$ $
	
	By Proposition \ref{prop:meas_disint},
	\[
	\mu_{\text{wp}}(U_X(2\epsilon)) = \int_{0}^\infty \eta_\lambda^r(U_X(2\epsilon)) \ dr.
	\]
	As a consequence of (\ref{eq:unit_speed_2}), $\varphi^\lambda_t(V) \subseteq U_X(2 \epsilon)$ for every $0 \leq t \leq \epsilon/C_2(\lambda)$. It follows that
	\[
	\int_{0}^\infty \eta_\lambda^r(U_X(2\epsilon)) \ dr \geq \int_0^{{\epsilon/C_2(\lambda)}} \eta_\lambda^{1+t}(\varphi^\lambda_t(V)) \ dt.
	\]
	Using (\ref{eq:slice_measure_bound}) we can bound 
	\[
	\int_0^{{\epsilon/C_2(\lambda)}} \eta_\lambda^{1+t}(\varphi^\lambda_t(V)) \ dt \geq 
	C_3(\lambda) \cdot \frac{\epsilon}{C_2(\lambda)} \cdot \eta_\lambda^{1}(V).
	\]
	By (\ref{eq:ball_meas_bound}),
	\[
	\mu_{\text{wp}}(U_X(2\epsilon)) \leq C_1 \cdot 2^{6g-6+2n} \cdot \epsilon^{6g-6+2n}.
	\]
	Putting things together we deduce
	\[
	\eta_\lambda^{1}(V) \leq \frac{C_1 \cdot 2^{6g-6+2n} \cdot C_2(\lambda)}{C_3(\lambda)} \cdot \epsilon^{6g-6+2n-1}.
	\]
	$ $
	
	By Theorem \ref{theo:smooth_length}, the function $C \colon \mathcal{ML}_{g,n} \to \mathbf{R}_{>0}$ which to every $\lambda \in \mathcal{ML}_{g,n}$ assigns the value
	\[
	C(\lambda):= \frac{C_1 \cdot 2^{6g-6+2n} \cdot C_2(\lambda)}{C_3(\lambda)}
	\]
	is continuous. Moreover, as $\mathcal{ML}_{g,n}(K(\epsilon_0)) \subseteq \mathcal{ML}_{g,n}$ is compact, the function $C$ attains a maximum on such set. Letting
	\[
	C:= \max_{\lambda \in \mathcal{ML}_{g,n}(K(\epsilon_0))} C(\lambda)
	\]
	finishes the proof.
\end{proof}
$ $


\bibliographystyle{amsalpha}


\bibliography{bibliography}

\providecommand{\bysame}{\leavevmode\hbox to3em{\hrulefill}\thinspace}
\providecommand{\MR}{\relax\ifhmode\unskip\space\fi MR }
\providecommand{\MRhref}[2]{%
  \href{http://www.ams.org/mathscinet-getitem?mr=#1}{#2}
}
\providecommand{\href}[2]{#2}
\begin{thebibliography}{ABEM12}

\bibitem[AA19]{AA19}
Francisco {Arana-Herrera} and Jayadev~S. {Athreya}, \emph{{Square-integrability
  of the Mirzakhani function and statistics of simple closed geodesics on
  hyperbolic surfaces}}, arXiv e-prints (2019), arXiv:1907.06287.

\bibitem[ABEM12]{ABEM12}
Jayadev Athreya, Alexander Bufetov, Alex Eskin, and Maryam Mirzakhani,
  \emph{Lattice point asymptotics and volume growth on {T}eichm\"uller space},
  Duke Math. J. \textbf{161} (2012), no.~6, 1055--1111. \MR{2913101}

\bibitem[{Ara}19]{AH19b}
Francisco {Arana-Herrera}, \emph{{Counting multi-geodesics on hyperbolic
  surfaces with respect to the lengths of individual components}}, In
  preparation, 2019.

\bibitem[Dan81]{Dani81}
S.~G. Dani, \emph{Invariant measures and minimal sets of horospherical flows},
  Invent. Math. \textbf{64} (1981), no.~2, 357--385. \MR{629475}

\bibitem[EM93]{EM93}
Alex Eskin and Curt McMullen, \emph{Mixing, counting, and equidistribution in
  {L}ie groups}, Duke Math. J. \textbf{71} (1993), no.~1, 181--209.
  \MR{1230290}

\bibitem[EMM19]{EMM19}
Alex {Eskin}, Maryam {Mirzakhani}, and Amir {Mohammadi}, \emph{{Effective
  counting of simple closed geodesics on hyperbolic surfaces}}, arXiv e-prints
  (2019), arXiv:1905.04435.

\bibitem[FM12]{FM11}
Benson Farb and Dan Margalit, \emph{A primer on mapping class groups},
  Princeton Mathematical Series, vol.~49, Princeton University Press,
  Princeton, NJ, 2012. \MR{2850125}

\bibitem[Ker83]{Ker83}
Steven~P. Kerckhoff, \emph{The {N}ielsen realization problem}, Ann. of Math.
  (2) \textbf{117} (1983), no.~2, 235--265. \MR{690845}

\bibitem[Ker85]{Ker85}
\bysame, \emph{Earthquakes are analytic}, Comment. Math. Helv. \textbf{60}
  (1985), no.~1, 17--30. \MR{787659}

\bibitem[KM96]{KM96}
D.~Y. Kleinbock and G.~A. Margulis, \emph{Bounded orbits of nonquasiunipotent
  flows on homogeneous spaces}, Sina\u{\i}'s {M}oscow {S}eminar on {D}ynamical
  {S}ystems, Amer. Math. Soc. Transl. Ser. 2, vol. 171, Amer. Math. Soc.,
  Providence, RI, 1996, pp.~141--172. \MR{1359098}

\bibitem[Mar04]{Mar04}
Grigoriy~A. Margulis, \emph{On some aspects of the theory of {A}nosov systems},
  Springer Monographs in Mathematics, Springer-Verlag, Berlin, 2004, With a
  survey by Richard Sharp: Periodic orbits of hyperbolic flows, Translated from
  the Russian by Valentina Vladimirovna Szulikowska. \MR{2035655}

\bibitem[Mir07a]{Mir07b}
Maryam Mirzakhani, \emph{Random hyperbolic surfaces and measured laminations},
  In the tradition of {A}hlfors-{B}ers. {IV}, Contemp. Math., vol. 432, Amer.
  Math. Soc., Providence, RI, 2007, pp.~179--198. \MR{2342816}

\bibitem[Mir07b]{Mir07a}
\bysame, \emph{Simple geodesics and {W}eil-{P}etersson volumes of moduli spaces
  of bordered {R}iemann surfaces}, Invent. Math. \textbf{167} (2007), no.~1,
  179--222. \MR{2264808}

\bibitem[Mir07c]{Mir07c}
\bysame, \emph{Weil-{P}etersson volumes and intersection theory on the moduli
  space of curves}, J. Amer. Math. Soc. \textbf{20} (2007), no.~1, 1--23.
  \MR{2257394}

\bibitem[Mir08a]{Mir08a}
\bysame, \emph{Ergodic theory of the earthquake flow}, Int. Math. Res. Not.
  IMRN (2008), no.~3, Art. ID rnm116, 39. \MR{2416997}

\bibitem[Mir08b]{Mir08b}
\bysame, \emph{Growth of the number of simple closed geodesics on hyperbolic
  surfaces}, Ann. of Math. (2) \textbf{168} (2008), no.~1, 97--125.
  \MR{2415399}

\bibitem[{Mir}16]{Mir16}
M.~{Mirzakhani}, \emph{{Counting Mapping Class group orbits on hyperbolic
  surfaces}}, ArXiv e-prints (2016).

\bibitem[MW02]{MW02}
Yair Minsky and Barak Weiss, \emph{Nondivergence of horocyclic flows on moduli
  space}, J. Reine Angew. Math. \textbf{552} (2002), 131--177. \MR{1940435}

\bibitem[PH92]{PH92}
R.~C. Penner and J.~L. Harer, \emph{Combinatorics of train tracks}, Annals of
  Mathematics Studies, vol. 125, Princeton University Press, Princeton, NJ,
  1992. \MR{1144770}

\bibitem[PS15]{Pap15}
A.~Papadopoulos and W.~Su, \emph{On the {F}insler structure of
  {T}eichm\"{u}ller's metric and {T}hurston's metric}, Expo. Math. \textbf{33}
  (2015), no.~1, 30--47. \MR{3310926}

\bibitem[Sar81]{Sar81}
Peter Sarnak, \emph{Asymptotic behavior of periodic orbits of the horocycle
  flow and {E}isenstein series}, Comm. Pure Appl. Math. \textbf{34} (1981),
  no.~6, 719--739. \MR{634284}

\bibitem[{Thu}98]{Thu98}
William~P. {Thurston}, \emph{{Minimal stretch maps between hyperbolic
  surfaces}}, arXiv Mathematics e-prints (1998), math/9801039.

\bibitem[Wol83]{Wol83}
Scott Wolpert, \emph{On the symplectic geometry of deformations of a hyperbolic
  surface}, Ann. of Math. (2) \textbf{117} (1983), no.~2, 207--234. \MR{690844}

\bibitem[Wol85]{Wol85}
\bysame, \emph{On the {W}eil-{P}etersson geometry of the moduli space of
  curves}, Amer. J. Math. \textbf{107} (1985), no.~4, 969--997. \MR{796909}

\bibitem[{Wri}18]{Wri18}
Alex {Wright}, \emph{{Mirzakhani's work on earthquake flow}}, arXiv e-prints
  (2018), arXiv:1810.07571.

\bibitem[{Wri}19]{Wri19}
\bysame, \emph{{A tour through Mirzakhani's work on moduli spaces of Riemann
  surfaces}}, arXiv e-prints (2019), arXiv:1905.01753.

\end{thebibliography}

$ $

\end{document}